\documentclass[12pt]{amsart}
\usepackage{amsmath}
\usepackage[active]{srcltx}
\usepackage{t1enc}
\usepackage[latin2]{inputenc}
\usepackage{verbatim}
\usepackage{amsmath,amsfonts,amssymb,amsthm}
\usepackage[mathcal]{eucal}
\usepackage{enumerate}
\usepackage[centertags]{amsmath}
\usepackage{graphics}
\numberwithin{equation}{section}

\newtheorem{theorem}{Theorem}[section]

\newtheorem{lemma}{Lemma}[section]

\newtheorem{definition}{Definition}[section]

\newtheorem{procedure}{Procedure}
\newtheorem{remark}{Remark}[section]
\numberwithin{equation}{section}

\def\r{{\rm r}}

\def\BO{{\rm BO}}

\def\supp{{\rm supp\,}}
\def\diam{{\rm diam\, }}
\def\dist{{\rm dist}}
\def\Gen{{\mathfrak {Gen}}}
\def\ZU{\ensuremath{\mathfrak U}}
\def\ZS{\ensuremath{\mathcal S}}
\def\ZM{\ensuremath{\mathfrak M}}

\def\ZB{\ensuremath{\mathfrak B}}
\def\ZA{\ensuremath{\mathcal A}}
\def\zA{\ensuremath{\mathfrak A}}
\def\ZZ{\ensuremath{\mathbb Z}}
\def\ZI{\ensuremath{\mathbb I}}
\def\ZN{\ensuremath{\mathbb N}}

\def\ZP{\ensuremath{\mathcal P}}
\def\ZD{\ensuremath{\mathcal D}}
\def\ZK{\ensuremath{\mathcal K}}
\def\ZR{\ensuremath{\mathbb R}}
\def\pr{\ensuremath{\mathrm {pr}}}
\def\ZL{\ensuremath{\mathcal L}}
\def\ZD{\ensuremath{\mathcal D}}
\def\ZH{\ensuremath{\mathcal H}}
\def\Ch{{\mathfrak {Ch}}}
\def\ZG{{\mathcal G\,}}

\newcommand {\e }[1]{(\ref{#1})}
\newcommand {\lem }[1]{Lemma \ref{#1}}

\newcommand {\df }[1]{Definition \ref{#1}}
\newcommand {\trm }[1]{Theorem \ref{#1}}

\newcommand {\sect }[1]{Section \ref{#1}}

\begin{document}
	\title{An abstract theory of singular operators}
	
	\author
	{Grigori A. Karagulyan}
	
	\address{G. A. Karagulyan, Faculty of Mathematics and Mechanics, Yerevan State University, Alex Manoogian, 1, 0025, Yerevan, Armenia}
	\email{g.karagulyan@ysu.am} 
	
	\date{}
	\maketitle
	\begin{abstract}
		We introduce a class of operators on abstract measure spaces, which unifies the Calder\'on-Zygmund operators on spaces of homogeneous type, the maximal functions, the martingale transforms and Carleson operators. We prove that such operators can be dominated by simple sparse operators with a definite form of the domination constant. Applying these estimates, we improve several results obtained by different authors in recent years.
	\end{abstract}
	\section{introduction}
	
	The study of weighted inequalities in Harmonic Analysis started in early 1970's. In 1972 Muckenhaupt \cite{Muck} proved  that the maximal function is bounded on $L^p(w)$ for $1<p<\infty$ if and only if the weight $w$ satisfies the $A_p$ condition
	\begin{equation}\label{h63}
	[w]_{A_p}=\sup_{I}\left(\frac{1}{|I|}\int_Iw\right)\left(\frac{1}{|I|}\int_Iw^{-1/(p-1)}\right)^{p-1}<\infty.
	\end{equation}
	One year later Hunt, Muckenhaupt and Wheeden \cite{HMW} established the same property for Hilbert transform. For the general Calder\'on-Zygmund operators weighted $L^p(w)$ bound was first proved by Coifman and Fefferman \cite{CoFe}. 
	
	In 1993 Buckley \cite{Buck}  discovered that the maximal function $M$ has the sharp estimate 
	\begin{equation}\label{z53}
	\|Mf\|_{L^p(w)\to L^p(w)}\le C\|w\|_{A_p}^{1/(p-1)},
	\end{equation}
	arising a similar problem for Calder\'on-Zygmund operators. Last fifteen years there has been an activity in the investigation of this problem. For the general Calder\'on-Zygmund operators the conjecture was the bound
	\begin{equation}\label{conj}
	\|Tf\|_{L^p(w)\to L^p(w)}\le C\|w\|_{A_p}^{\max\{1,1/(p-1)\}}.
	\end{equation} 
	An extrapolation theorem proved in \cite{DGPP} reduced the conjecture to the case $p=2$ so the inequality \e {conj} became known as $A_2$ conjecture. After being established for several particular operators first (\cite{PeVo}, \cite{Pet1}, \cite{Pet2}, \cite{PTV}, \cite{Hyt7}), in 2010 Hyt\"onen \cite{Hyt1} proved $A_2$ conjecture for the general Calder\'on-Zygmund operators.

	Series of recent works are motivated on domination of Calder\'on-Zygmund operators by very simple sparse operators (\cite {Ler2}-\cite{Ler6}, \cite{CoRe}, \cite{Lac1}, \cite{Hyt2}-\cite{Hyt6}). From such results in particular follows \e {conj}, since the $A_2$ bound is very easy to establish for the sparse operators (\cite{Ler4}, \cite{Hyt3}, \cite{Moen1}, \cite{Moen2}). A domination of the classical Calder\'{o}n-Zygmund operators on $\ZR^n$ by sparse operators was discovered by Lerner \cite{Ler4}. Applying this domination he gave a simplified proof of Hyt\"onen $A_2$ theorem.  Lacey \cite{Lac1} proved a pointwise domination theorem for more general $\omega$-Calder\'{o}n-Zygmund operators (see \e {h75}-\e {h80}) with  $\omega$ satisfying the Dini condition 
	\begin{equation}\label{z40}
	\int_0^1\frac{\omega(t)}{t}dt<\infty,
	\end{equation} 
	deriving weighted bound \e {conj} for such operators too. Lacey's result was a stronger version of another pointwise bound independently proved by Conde-Alonso, Rey \cite{CoRe} and Lerner-Nazarov \cite{LerNaz}. Moreover, Lacey's inequality only assumes the Dini condition, while prior approaches \cite {CoRe, LerNaz} require $1/t$ in the Dini integral be replaced by $(\log 2/t)/t$.  Hyt\"onen, Roncal and Tapiola \cite{Hyt6} elaborated the proof of Lacey \cite{Lac1} to get a precise linear dependence of the domination constant on the characteristic numbers of the operator. Lerner \cite{Ler6} gave a simple proof of Lacey-Hyt\"onen-Roncal-Tapiola theorem.
	
	Anderson and Vagharshakyan \cite{Vag2} proved $A_2$ theorem for the Calder\'{o}n-Zygmund operators in general spaces of homogeneous type with modulus of continuity $\omega(t)=t^\alpha$, $\alpha>0$. 
	
	In late 1970's, several authors considered also martingale analogs of the $A_p$ theory. For
	instance, Izumisawa-Kazamaki \cite{IzKa}, proved a variant of the Muckenhoupt maximal function
	result \cite{Muck} in this setting. When it came to martingale transforms, the distinction between
	the homogeneous and non-homogeneous cases was already recognized by these authors.
	Nevertheless, norm inequalities for martingale transforms were proved by Bonami-Lépingle
	\cite{BoLe}. The $A_2$ theorem for martingale transform was proved recently by Thiele, Treil and Volberg \cite {TTV}.
	Lacey \cite{Lac1} gave a self-contained, short and elementary proof of this theorem. Moreover, he  established a pointwise domination theorem for martingale transforms too. 
	
	Grafakos-Martell-Soria \cite{GMS} and  Di Plinio-Lerner \cite{PL} considered weighted estimates for maximal modulations of Calder\'{o}n-Zygmund operators on $\ZR^n$, in particular, for Carleson or Walsh-Carleson operators. The paper \cite{GMS} establishes weighted norm control of the Carleson operators by the maximal function. In \cite{PL} the authors proved weighted norm estimates of Carleson and Walsh-Carleson operators with explicit dependence of the constants on $A_p$-characteristics of the weight.   
	
	In this paper we introduce so called $\BO$ (bounded oscillation) operators on abstract measure spaces. Those operators unify the Calder\'on-Zygmund operators and maximal functions in general homogeneous spaces, martingale transforms (non-homogeneous case) and Carleson operators. The definition of $\BO$ operators is motivated by the papers \cite{Kar1}, \cite {Kar2}, where some exponential estimates for Calder\'on-Zygmund and other related operators were proved. We shall prove that those operators have pointwise domination by sparse operators and then satisfy the bound \e {conj}. We derive variety of other properties of $\BO$ operators significant for their further investigations. 
	
	To define $\BO$ operators we introduce a concept of ball-basis for an abstract measure space, which is a family of measurable sets holding some common properties of $d$-dimensional balls on $\ZR^d$ and their analogues in related theories (martingales, dyadic analysis). 
	
	\begin{definition} Let $(X,\ZM,\mu)$ be a measure space. A family of sets $\ZB\subset \ZM$ is said to be a ball-basis if it satisfies the following conditions: 
		\begin{enumerate}
			\item[B1)] $0<\mu(B)<\infty$ for any ball $B\in\ZB$.
			\item[B2)] For any points $x,y\in X$ there exists a ball $B\ni x,y$.
			\item[B3)] If $E\in \ZM$, then for any $\varepsilon>0$ there exists a finite or infinite sequence of balls $B_k$, $k=1,2,\ldots$, such that
			\begin{equation*}	
			\mu\left(E\bigtriangleup \bigcup_k B_k\right)<\varepsilon.
			\end{equation*} 
			\item[B4)] For any $B\in\ZB$ there is a ball $B^*\in\ZB $ (called {\rm hull} of $B$) satisfying the conditions
			\begin{align}
			&\bigcup_{A\in\ZB:\, \mu(A)\le 2\mu(B),\, A\cap B\neq\varnothing}A\subset B^*,\label{h12}\\
			&\qquad\qquad\mu(B^*)\le \ZK\mu(B),\label{h13}
			\end{align}
			where $\ZK$ is a positive constant.
		\end{enumerate}
		
	\end{definition}
	One can easily check that the family of Euclidean balls in $\ZR^n$ forms a ball-basis. Moreover, we will see below (see \trm {T18}) that if the family of metric balls in spaces of homogeneous type satisfies the density condition (see \df {density}), then it is a ball-basis too. The martingale basis considered in \sect{S8} is an example of ball-bases having non-doubling property. 
	\begin{definition}
		Let $1\le r<\infty$, $(X,\ZM,\mu)$ be a measure space and $L^0(X)$ be the linear space of functions (include non-measurable functions) on $X$. An operator 
		\begin{equation*}
		T:L^r(X)\to L^0(X)
		\end{equation*}
		is said to be subadditive if
		\begin{align*}
		&|T(\lambda\cdot f)(x)|=|\lambda|\cdot |Tf(x)|,\quad \lambda\in\ZR,\\
		&|T(f+g)(x)|\le |Tf(x)|+|Tg(x)|.
		\end{align*}
	\end{definition}
	\begin{remark}
		As we will see below in the definitions of some operators (maximal function, $T^*$) some non-measurable functions can appear. To apply the results of the paper for such operators we allow non-measurably of the images in the definition of general subadditive operators. The definitions of $L^p$ (weak-$L^p$) norms and some standard inequalities that we need for non-measurable functions will be stated in the next section.
	\end{remark}
	Let $1\le r<\infty$ be a fixed number in Sections 1-4. For $f\in L^r(X)$ and $B\in\ZB$ we set 
	\begin{equation*}
	\langle f\rangle_{B,r}=\left(\frac{1}{\mu(B)}\int_{B}|f|^r\right)^{1/r},\quad \langle f \rangle^*_{B,r}=\sup_{A\in \ZB:A\supset B}\langle f\rangle_{A,r}.
	\end{equation*}
	In the case $r=1$ for those quantities it will be used the notations $\langle f\rangle_{B}$ and $\langle f \rangle^*_{B}$ respectively (Sections 5-8).
	\begin{definition}
		We say that a subadditive operator $T$ is a bounded oscillation operator with respect to a ball-basis $\ZB$ if 
		\begin{enumerate}
			\item [T1)](Localization) for every $B\in \ZB$ we have
			\begin{multline}\label{d7}
			\sup_{x,x'\in B,\,f\in L^r(X) }\frac{|T(f\cdot \ZI_{X\setminus B^*})(x)-T(f\cdot \ZI_{X\setminus B^*})(x')|}{\langle f \rangle^*_{B,r}}\\
			\le\ZL_1=\ZL_1(T)<\infty,
			\end{multline}
			\item [T2)](Connectivity) for any $A\in \ZB$ ($A^*\neq X$) there exists a ball $B\supsetneq A$ (i.e. $B\supset A$, $B\neq A$) such that
			\begin{equation}\label{d20}
			\sup_{x\in A,\,f\in L^r(X) }\frac{|T(f\cdot \ZI_{B^*\setminus A^*})(x)|}{\langle f\rangle_{B^*,r}}\le \ZL_2=\ZL_2(T)<\infty,
			\end{equation}
		\end{enumerate}
		where $\ZL_1$ and $\ZL_2$ are the least constants satisfying \e {d7} and \e {d20} respectively. The family of all bounded oscillation operators with respect to a ball-basis $\ZB$ will be denoted by $\BO_\ZB$ or simply $\BO$.
	\end{definition}
	It will be proved below (\trm {T4}) that if the ball-basis $\ZB$ satisfies the doubling condition, then the localization property implies connectivity. 
	\begin{definition}
		A collection of balls $\ZS\subset \ZB$ is said to be sparse or $\gamma$-sparse if for any $B\in \ZS$ there is a set $E_B\subset B$ such that $\mu(E_B)\ge \gamma\mu(B)$ and the sets $\{E_B:\, B\in \ZS\}$ are pairwise disjoint, where  $0<\gamma<1$ is a constant.   
	\end{definition}
	Given family of balls $\ZS$ we associate the operator 
	\begin{equation*}
	\ZA_{\ZS,r}f(x)=\sum_{A\in \ZS}\left\langle f\right\rangle_{A,r}\cdot\ZI_A(x).
	\end{equation*}
	If $\ZS$ is a sparse collection of balls, then we say $\ZA_{\ZS,r}$ is a sparse operator.  In the case $r=1$ we will simply write $\ZA_{\ZS}$. Further, positive constants depending on $\ZK$ (see \e {h13}) will be called admissible constants and the relation $a\lesssim b$ will stand for $a\le c\cdot b$, where $c>0$  is admissible. We write $a\sim b$ if the relations  $a\lesssim b$ and $b\lesssim a$ hold at the same time.
	\begin{theorem}\label{T1}
		Let an operator $T\in\BO_\ZB(X)$ satisfy weak-$L^r$ inequality. Then for any function $f\in L^r(X)$ and a ball $B\in\ZB$ there exists a family of balls $\ZS$, which is a union of two $\gamma$-sparse collections and 
		\begin{equation}\label{z34}
		|Tf(x)|\lesssim (\ZL_1+\ZL_2+\|T\|_{L^r\to L^{r,\infty}})\cdot \ZA_{\ZS,r} f(x),\text { a.e. } x\in B,
		\end{equation}
		where $\ZL_1$ and $\ZL_2$ are the constants \e {d7}, \e {d20} and $0<\gamma<1$ is an admissible constant.
	\end{theorem}
	\begin{definition}
		For $T\in \BO_\ZB$ define
		\begin{equation*}
		T^*f(x)=\sup_{B\in \ZB:\, x\in  B}|T(f\cdot \ZI_{X\setminus B^*})(x)|.
		\end{equation*}	
	\end{definition}
	
	We shall prove below (\trm {T6}) that if $T\in\BO_\ZB$ satisfies weak-$L^r$ estimate, then $T^*\in\BO_\ZB$ and satisfies weak-$L^r$ bound too. So from \trm {T1} we will immediately get 
	\begin{theorem}\label{T0}
		If $T\in\BO_\ZB(X)$ satisfies weak-$L^r$ inequality, then for any function $f\in L^r(X)$ and a ball $B\in\ZB$ there exists a family of balls $\ZS$, which is a union of two $\gamma$-sparse collections and 
		\begin{equation*}
		|T^*f(x)|\lesssim (\ZL_1(T)+\ZL_2(T)+\|T\|_{L^r\to L^{r,\infty}})\cdot\ZA_{\ZS,r} f(x),\text { a.e. } x\in B.
		\end{equation*}
	\end{theorem}
	\trm {T20}, \trm {T5} and \trm {T19} below prove that the $\omega$-Calder\'{o}n-Zygmund operators on spaces of homogeneous type, the martingale transforms and the maximal functions are $\BO$ operators and so they all satisfy the estimate \e {z34}. Moreover, combining \e {z34} with the weighted bounds for sparse operators, we obtain $A_2$ theorems for all these operators. Hence \trm {T1} and \trm {T0} cover the above stated results concerning the weighted bounds and the domination by sparse operators. Lacey-Hyt\"onen-Roncal-Tapiola (\cite{Lac1}, \cite{Hyt6}) theorem is a version of \trm {T0} for the $\omega$-Calder\'{o}n-Zygmund operators on $\ZR^n$ with the Dini condition. Lacey \cite{Lac1} domination theorem for martingale transforms is another case of inequality \e {z34}. In $\ZR^n$ and in general spaces of homogeneous type, where a doubling condition holds, the proofs of such dominations are based on dyadic decomposition theorems (\cite{Hyt3}, \cite{HyKa}). In the case of martingale transforms (\cite{TTV}) instead of dyadic decomposition the properties of martingale basis is essentially used. Since a general ball-basis does not always satisfy the doubling condition (the typical example is the martingale basis), there is no dyadic decomposition in general. So our method of proof of \trm {T1} is different. It is direct and partially based on the papers \cite{Kar1} and \cite{Kar2}. 
	
	In the last sections we prove several weighted bounds for sparse operators to get $A_2$ theorems for some $\BO$ operators. The method of proofs of such theorems are based on a duality argument developed in the papers \cite{Ler6}, \cite{CoRe}, \cite{Hyt3}, \cite{Lac1}.
	
	Note that \trm {T1} also imply exponential integrability results of our papers \cite{Kar1}, \cite {Kar2} proved for the Calder\'{o}n-Zygmund operators on Euclidean spaces and for the partial sums of Walsh-Fourier series.
	
	Applying \trm {T1} for the maximal function corresponding to a general ball-basis, we do not get the full weighted estimate \e {z53}, which is known to be optimal for the maximal function in Euclidean spaces (\cite{Buck}). In the general case the optimality only occurs when $1<p\le 2$. So \e {z34} do not cover some estimates that has the maximal function. The reason is that the maximal function has some extra properties that the general $\BO$ operators do not have. An example of such a property is $L^\infty$-bound. 
	
	In the last section we prove that the maximal modulation of a $\BO$ operator is also $\BO$ operator, deriving pointwise sparse domination for maximal modulated $\BO$ operators too.  
	
	\section{Outer measure and $L^p$-norms of non-measurable functions}
	Let $(X,\ZM,\mu)$ be a measure space. Define the outer measure of a set $E\subset X$ by
	\begin{equation*}
	\mu^*(E)=\inf_{F\in \ZM:\, F\supset E}\mu(F).
	\end{equation*}
	We say two sets $A,B\subset X$ (not necessarily measurable) satisfy the relation $A\sim B$ if  
	\begin{equation*}
	\mu^*(A\bigtriangleup B)=0.
	\end{equation*}
	Denote by $\bar{\ZM}$ the family of sets in $X$, which are equivalent to a measurable set. It is clear that $\bar \ZM$ is $\sigma$-algebra. For a given set $E\subset X$ we define
	\begin{equation*}
	\bar E=\bigcap_{F\in\ZM:\, F\supset E}F.
	\end{equation*}
	Observe that
	\begin{equation*}
	\bar E\in \bar \ZM,\quad \mu^*(\bar E)=\mu^*(E).
	\end{equation*}
	For any function $f\in L^0(X)$ we denote
	\begin{align*}
	&G_f(t)=\{x\in X:\, |f(x)|>t\},\\
	& \lambda_f(t)=\mu^*\left(G_f(t)\right).
	\end{align*}
	Observe that the function
	\begin{equation*}
	\bar f(x)=\inf\{t\ge 0:\, x\in \overline {G_f(t)} \}
	\end{equation*}
	is positive, $\bar{\ZM}$-measurable and $|f(x)|\le \bar f(x)$. Besides, $\bar f$ is the smallest $\bar{\ZM}$-measurable positive function dominating $|f|$. Namely, if $g(x)$ is $\bar{\ZM}$-measurable and satisfies $|f(x)|\le g(x)$, then $\bar f(x)\le g(x)$. For arbitrary $f\in L^0(X)$ we define
	\begin{align*}
	&\|f\|_{L^p}=\|\bar f\|_{L^p}=\left(p\int_0^\infty t^{p-1}\lambda_f(t)dt\right)^{1/p},\\
	&\|f\|_{L^p\to L^{p,\infty}}=\|\bar f\|_{L^p\to L^{p,\infty}}=\sup_{t>0}t(\lambda_f(t))^{1/p}.
	\end{align*}
	
	\begin{definition}
		We say a subadditive operator $T$ satisfies weak-$L^p$ or strong-$L^p$ estimate if
		\begin{align*}
		&\|T\|_{L^p\to L^{p,\infty}}=\sup_{t>0,\,f\in L^p(X)}\frac{t\cdot ( \lambda_{Tf}(t))^{1/p}}{\|f\|_{L^p}}<\infty,\\
		&\|T\|_{L^p}=\sup_{f\in L^p(X)}\frac{\|Tf\|_{L^p}}{\|f\|_{L^p}}<\infty,
		\end{align*}
		respectively. 
	\end{definition}
	One can easily check that the standard triangle and  H\"{o}lder inequalities as well as the Marcinkiewicz interpolation theorem hold in such setting of $L^p$ norms. We will need the following case of the interpolation theorem.
	\begin{theorem}[Marcinkiewicz interpolation theorem, \cite {Zyg}, ch. 12.4]\label{M} If a subadditive  operator $T$ satisfies the weak-$L^1$ estimate and 
		the strong-$L^\infty$ estimate, then for $1<p<\infty$ it holds
		\begin{equation*}
		\|T\|_{L^p}\le c_p (\|T\|_{L^1\to L^{1,\infty}})^{1/p}\times (\|T\|_{L^\infty})^{1/q}.
		\end{equation*}
	\end{theorem}
	\section{Some properties of ball-basis}
	
	Let $\ZB$ be a ball-basis for the measure space $(X,\ZM,\mu)$. From B4) condition it follows that if balls $A,B$ satisfy $\mu(A)\le 2\mu(B)$, then $A\subset B^*$. This property will be called two balls relation.  Hull levels of a given ball $B\in \ZB$ will be denoted by
	\begin{equation*}
	B^{[0]}=B,\quad B^{[n+1]}=\left(B^{[n]}\right)^*.
	\end{equation*} 
	By property B4) we have $\mu(B^{[n+1]})\le \ZK\mu(B^{[n]})$. Applying this inequality consecutively, we get
	\begin{equation}\label{a33}
	\mu(B^{[n]})\le \ZK^n\mu(B),\quad n\ge 0.
	\end{equation}
	We say a set $E\subset X$ is bounded if $E\subset B$ for some $B\in\ZM$.
	\begin{lemma}\label{L1-1}
		Let $(X,\ZM,\mu)$ be a measure space and the family of sets $\ZB\subset \ZM$ satisfy B4)-condition. If $E\subset X$ is bounded and $\ZG $ is a family of balls with
		\begin{equation*}
		E\subset \bigcup_{G\in \ZG}G,
		\end{equation*}
		then there exists a finite or infinite sequence of pairwise disjoint balls $G_k\in \ZG$ such that
		\begin{equation}\label{c1}
		E \subset \bigcup_k G_k^{[1]}.
		\end{equation}
	\end{lemma}
	\begin{proof}
		The boundedness implies $E\subset B$ for some $B\in \ZB$. If there is a ball $G\in\ZG$ so that $G\cap B\neq \varnothing$ and $\mu(G)> \mu(B)$, then by two balls relation we have $E\subset B\subset G^{[1]}$. Thus the desired sequence can be formed by a single element $G$. Hence we can suppose that any element $G\in\ZG$ satisfies the conditions $G\cap B\neq \varnothing$ and $\mu(G)\le \mu(B)$. Therefore we get
		\begin{equation*}
		\bigcup_{G\in\ZG}G\subset B^{[1]}.
		\end{equation*}
		Take $G_1$ to be a ball from $\ZG$ satisfying 
		\begin{equation*}
		\mu(G_1)> \frac{1}{2}\sup_{G\in\ZG}\mu(G).
		\end{equation*}
		Then, suppose by induction we have already chosen elements $G_1,\ldots,G_k$ from $\ZG$. Take $G_{k+1}\in \ZG$  disjoint with the balls $G_1,\ldots,G_k$ and satisfying
		\begin{equation*}
		\mu(G_{k+1})> \frac{1}{2}\sup_{G\in \ZG:\,G\cap G_j=\varnothing,\,j=1,\ldots,k}\mu(G).
		\end{equation*}
		If for some $n$ we will not be able to determine $G_{n+1}$ the process will stop and we will get a finite sequence $G_1,G_2,\ldots, G_n$. Otherwise our sequence will be infinite. We shall consider the infinite case of the sequence (the finite case can be done similarly). Since the balls $G_n$ are pairwise disjoint and $G_n\subset B^{[1]}$, we have $\mu(G_n)\to 0$. Take an arbitrary $G\in\ZG$ such that $G\neq G_k$, $k=1,2,\ldots $. Let $m$ be the smallest integer such that
		\begin{equation*}
		\mu(G)>\frac{1}{2}\mu(G_{m+1}).
		\end{equation*}
		Observe that we have
		\begin{equation*}
		G \cap G_j\neq\varnothing
		\end{equation*}
		for some of $1\le j\le m$, since otherwise $G$ had to be chosen instead of $G_{m+1}$. Besides, we have $\mu(G)\le 2\mu(G_{j})$, which implies $G\subset G_{j}^{[1]}$. Since $G\in \ZG$ was taken arbitrarily, we get \e {c1}.
	\end{proof}
	\begin{lemma}\label{L1}
		Let $(X,\ZM,\mu)$ be a measure space with a ball-basis $\ZB$. If balls $B\in \ZB$, $G_k\in\ZB$, $k=1,2,\ldots$, satisfy the relation
		\begin{equation}\label{z18}
		G_k\cap B\neq\varnothing,\quad \mu(G_k)\to r=\sup_{A\in \ZB}\mu(A),
		\end{equation}
		then 
		\begin{equation*}
		X\subset \bigcup_k G_k^{[1]}.
		\end{equation*}
		Moreover, for any ball $A\in\ZB$ we have $A\subset G_k^{[1]}$ for some $k\ge k_0$. 
	\end{lemma}	
	\begin{proof}
		Since by B2)-condition every point $x\in X$ is in some ball, it is enough to prove the second part of the theorem. So let $A\in\ZB$. Chose points $x\in A$, $y\in B$. According to B2)-condition there is $C\in\ZB$ such that $x,y\in C$. Let $G$ be one of the balls $A,B$ and $C$, which has a biggest measure. Applying two ball relation twice, one can easily check that 
		\begin{equation*}
		A\cup B\cup C\subset G^{[2]}.
		\end{equation*}	 
		From \e {z18} we find an integer $k_0$ such that $\mu(G_k)>\mu(G^{[2]})/2$ for $k>k_0$. Therefore, since $G_k\cap G^{[2]}\neq\varnothing$, we get $A\subset G^{[2]}\subset G_k^{[1]}$, $k>k_0$.
	\end{proof}
	
	\begin{definition}\label{density}
		For a set $E\in \ZM $ a point $x\in E$ is said to be density point if for any $\varepsilon>0$ there exists a ball $B$ such that
		\begin{equation*}
		\mu(B\cap E)>(1-\varepsilon )\mu(B).
		\end{equation*} 
		We say a measure space $(X,\ZM,\mu)$ satisfies the density property if for any measurable set $E$ almost all points $x\in E$ are density points. 
	\end{definition}
	\begin{lemma}\label{L20}
		Let $(X,\ZM,\mu)$ be a measure space. If a family of measurable sets $\ZB$ satisfies the density property and B4)-condition, then for any bounded measurable set $E$, $\mu(E)>0$, and $\varepsilon>0$ there is a sequence of balls $\{B_k\}$ such that
		\begin{equation}\label{h81}
		\mu\left(\bigcup_{k}B_k\setminus E\right)<\varepsilon,\quad \mu\left(E\setminus \bigcup_{k}B_k\right)<\alpha \mu(E)
		\end{equation}
		where $0<\alpha<1$ is an admissible constant. 
	\end{lemma}
	\begin{proof}
		Applying the density property, one can find a family of balls $\ZB$ satisfying
		\begin{align*}
		&E\subset \bigcup_{B\in\ZB}B \text { a.s.},\\
		&\mu(B\cap E)>(1-\delta )\mu(B),\quad B\in\ZB,
		\end{align*}
		where 
		\begin{equation*}
		\delta=\min\left\{\frac{\varepsilon  }{2\mu(E)},1/2\right\}.
		\end{equation*}
		Then, apply \lem {L1-1} we get a subfamily of pairwise disjoint balls $B_k$ with \e {c1}. Thus we have
		\begin{align*}
		\mu\left(\bigcup_{k}B_k\setminus E\right)&=\sum_k\mu(B_k\setminus E)\\
		&<\frac{\delta}{1-\delta} \sum_k\mu(B_k\cap E)\le 2\delta\mu(E)\le \varepsilon.
		\end{align*}
		On the other hand 
		\begin{align*}
		\mu\left(E\setminus \bigcup_{k}B_k\right)&=\mu(E)-\sum_k\mu\left(E\cap B_k\right)\\
		&\le \mu(E)-(1-\delta)\sum_k\mu\left(B_k\right)\\
		&\le \mu(E)-\frac{1-\delta}{\ZK}\mu\left(\bigcup_{k}B_k^{[1]}\right)\\
		&\le \mu(E)-\frac{1}{2\ZK}\mu\left(\bigcup_{B\in\ZB}B\right)\\
		&\le \mu(E)-\frac{1}{2\ZK}\mu(E)\\
		&=\left(1-\frac{1}{2\ZK}\right)\mu(E).
		\end{align*}
		So conditions \e {h81} are satisfied with a constant $\alpha=1-1/2\ZK$.
	\end{proof}
	Observe that if B3)-condition holds for the bounded measurable sets, then it holds for all the measurable sets. Indeed, according \lem {L1}, there is a sequence of balls $G_k$ such that $X=\cup_kG_k$. This implies that any measurable set $E$ can be written as a countable union of bounded measurable sets $E=\cup_kE_k$. Apply B3)-condition to each set $E_k$ with an approximation number $\varepsilon/2^k$. The union of all obtained approximating balls will give an $\varepsilon$-approximation of $E$.    
	\begin{lemma}\label{L12}
		Let $(X,\ZM,\mu)$ be a measurable set. If a family of measurable sets $\ZB$ satisfies B4)-condition, then it will  satisfy B3)-condition if and only if the density condition holds.
	\end{lemma}
	\begin{proof}
		Let $\ZB$ satisfy B4) and density conditions and $E$ be a measurable set. The remark stated before the lemma allows us to suppose that $E$ is bounded. Applying \lem {L20} consecutively, we can find sequences of balls $\ZB_k$, $k=1,2,\ldots,$ such that 
		\begin{align}
		&\mu\left(\bigcup_{B\in\ZB_k}B\setminus E\right)<\frac{\varepsilon}{2^k},\quad k\ge 1,\label{h83}\\
		&\mu\left(E\setminus \bigcup_{B\in\cup_{j=1}^{k}\ZB_j}B\right)<\alpha \mu\left(E\setminus \bigcup_{B\in\cup_{j=1}^{k-1}\ZB_j}B\right),\quad k\ge 1,\label{h82}
		\end{align} 
		where in the case $k=1$ the right hand side of \e {h82} is assumed to be $E$. Then we denote $\ZB=\cup_{j=1}^{\infty }\ZB_j $. From \e {h82} it easily follows that
		\begin{equation*}
		E\subset \bigcup_{B\in\ZB}B\text { a.s.},
		\end{equation*} 
		while from \e {h83} we obtain
		\begin{equation*}
		\mu\left(\left(\bigcup_{B\in\ZB}B\right)\setminus E\right)\le \sum_{k=1}^\infty\mu\left(\left(\bigcup_{B\in\ZB_k}B\right)\setminus E\right)<\varepsilon.
		\end{equation*}
		To prove the second part of lemma let $\ZB$ satisfy B4) and B3) conditions. Suppose to the contrary $\ZB$ does not have the density property. That is, there exists a number $\alpha$, $0<\alpha<1$, a set $E\in \ZM$ together with its subset $F\subset E$, $\mu^*(F)>0$, such that for any $x\in F$ and $B\in\ZB$ with $x\in B$ we have
		\begin{equation}\label{z38}
		\mu(B\setminus E)>\alpha\mu(B).
		\end{equation}
		By B3)-condition for any $\varepsilon >0$ it can be found a sequence of balls $B_k$, $k=1,2,\ldots $, such that
		\begin{equation}\label{z42}	
		\mu\left(\bar F\bigtriangleup \bigcup_k B_k\right)<\varepsilon.
		\end{equation}
		We can suppose that $\mu(B_k\cap \bar F)>0$ for each $B_k$. Observe that it implies $B_k\cap F\neq\varnothing$. Indeed, suppose to the contrary $B_{k_0}\cap F=\varnothing$. Then we get $F\subset \bar F\setminus B_{k_0}$ and so a contradiction $\mu^*(F)\le \mu^*(\bar F\setminus B_{k_0})<\mu^*(\bar F)$. Thus by \e {z38} we get
		\begin{equation}\label{z41}
		\mu(B_k\setminus \bar F)\ge \mu(B_k\setminus E)>\alpha \mu(B_k),\quad k=1,2,\ldots.
		\end{equation}
		Applying \lem {L20}, we find a subsequence of pairwise disjoint balls $\tilde B_k$, $k=1,2,\ldots$, such that
		\begin{equation*}
		\mu\left(\bar F\setminus\bigcup_k\tilde B_k^{[1]}\right)<\varepsilon.
		\end{equation*} 
		Thus, applying B4)-condition, \e {z42} and \e {z41}, we obtain
		\begin{align*}
		\mu^*(\bar F)&\le \mu\left(\bigcup_k\tilde B_k^{[1]}\right)+\varepsilon\le \ZK\sum_k\mu(\tilde B_k)+\varepsilon\le 
		\frac{\ZK}{\alpha}\sum_k\mu(\tilde B_k\setminus \bar F)+\varepsilon\\
		&= \frac{\ZK}{\alpha}\mu\left(\bar F\bigtriangleup \bigcup_k B_k\right)+\varepsilon<\varepsilon\left(\frac{\ZK}{\alpha}+1\right).
		\end{align*}
		Choosing enough small $\varepsilon $, we get $\mu^*(F)=\mu^*(\bar F)=0$ and so a contradiction.

	\end{proof}
	We say a measurable set $E$ is almost surely subset of a measurable set $F$ if $\mu(E\setminus F)=0$. We denote this relation by
	\begin{equation*}
	E\subset F\text{ a.s.. }
	\end{equation*}
	\begin{lemma}\label{L3}
		For any bounded measurable set $E\in \ZM$ there exists a sequence of balls $B_k$, $k=1,2,\ldots,$ such that
		\begin{align}
		&E\subset \bigcup_{k}B_k\text { a.s. }\label{a91}\\
		&\sum_k\mu(B_k)\le 2\ZK\mu(E).\label{a92}
		\end{align}	
	\end{lemma}
	\begin{proof}
		Then observe that applying B3)-condition consecutively, for a given measurable set $E$ and $\varepsilon>0$ one can find a countable family of balls $\zA$ such that 
		\begin{align}
		&E\subset  \bigcup_{A\in \zA}A\text { a.s. }\label{h107}\\
		&\mu\left(\bigcup_{A\in \zA}A\right)<(1+\varepsilon)\mu(E).\label{a96}
		\end{align}
		Applying \lem {L1-1}, we find a pairwise disjoint collection $\{A_j\}\subset \zA$ such that 
		\begin{equation*}
		E\subset \bigcup_{j}{A_j}^{[1]}\text { a.s. }.
		\end{equation*} 
		The B4)-condition and \e {a96} yield
		\begin{align*}
		\sum_{j}\mu\left({A_j}^{[1]}\right)&\le \ZK\sum_{j}\mu\left(A_j\right)=\ZK\mu\left(\bigcup_{j}A_j\right)\\
		&\le \ZK\mu\left(\bigcup_{A\in \zA}A\right)\\
		&<2\ZK\mu(E).
		\end{align*}
		Define $B_k=A_k^{[1]}$, one can easily check that \e {a91} and \e {a92} are satisfied.  
	\end{proof}

	We denote by $\# A$ the cardinality of a finite set $A$.
	\begin{lemma}\label{L1-2} Let $A\in \ZB$ and $\ZG$ be a family of pairwise disjoint balls such that each $G\in \ZG$ satisfies the relations 
		\begin{align}
		&G^{[1]}\cap A \neq \varnothing,\label{a36}\\
		& 0<c_1\le \mu(G)\le c_2,\label{a41}
		\end{align}
		with some positive constants $c_1,c_2$. Then the number of elements in $\ZG$ is finite and satisfies the bound
		\begin{equation*}
		\#\ZG\le \frac{\min\{\ZK^3c_2,\ZK\mu(A)\}}{c_1}.
		\end{equation*}
	\end{lemma}
	\begin{remark}
		One can easily check that this lemma implies a similar lemma with the condition $G^{[2]}\cap A \neq \varnothing$ instead of \e {a36}.
	\end{remark}
	\begin{proof} Suppose $G_1,G_2,\ldots, G_N$ are some elements from $\ZG$. We can assume that 
		\begin{equation}\label{a42}
		\mu(G_1^{[1]})\ge \mu(G_i^{[1]})
		\end{equation}
		for each $1\le j\le N$. If $\mu(A)\ge \mu(G_1^{[1]})$, then from \e {a36} and B4)-condition we get 
		\begin{equation*}
		\bigcup_{1\le j\le N} G_k\subset \bigcup_{1\le j\le N}G_j^{[1]}\subset A^{[1]}.
		\end{equation*}
		Thus, since $G_k$ are pairwise disjoint, from \e {a41} we obtain
		\begin{equation*}
		N\cdot c_1\le \mu\left(\bigcup_{1\le j\le N} G_k\right)\le \mu(A^{[1]})\le \ZK\mu(A)
		\end{equation*}
		that is 
		\begin{equation}\label{a45}
		N\le \frac{\ZK\mu(A)}{c_1}.
		\end{equation}
		In the case $\mu(A)< \mu(G_1^{[1]})$ we get $A\subset G_1^{[2]}$ and therefore by \e {a36},
		\begin{equation*}
		G_j^{[1]}\cap G_1^{[2]}\neq\varnothing,\quad 1\le j\le N.
		\end{equation*}
		Thus, applying two balls relation and \e {a42}, we obtain
		\begin{equation*}
		\bigcup_{1\le j\le N}G_j^{[1]}\subset G_1^{[3]}.
		\end{equation*}
		Then, again using \e {a41} and \e {a33}, we get
		\begin{equation}\label{a94}
		N\cdot c_1\le \mu\left(\bigcup_k G_k\right)\le \mu(G_1^{[3]})\le \ZK^3\mu(G_1)\le \ZK^3 c_2.
		\end{equation}
		Combination of \e {a45} and \e {a94} completes the proof of lemma.
	\end{proof}
	\section {Preliminary properties of bounded oscillation operators}
	In this section we derive some preliminary properties of $\BO$ operators. Let $(X,\ZM,\mu)$ be a measure space with a ball-basis $\ZB$ and $1\le r<\infty$. We will need weak-$L^r$ inequality of the maximal operator 
	\begin{equation}\label{1-1}
	M_r f(x)=\sup_{B\in \ZB:\, x\in B} \left(\frac{1}{\mu(B)}\int_B|f(t)|^rd\mu(t)\right)^{1/r}
	\end{equation}
	associated with a ball-basis $\ZB$.  The maximal operator corresponding to $r=1$ will be denoted by $M$.
	\begin{theorem}\label{T1-1}
		The maximal operator \e {1-1} satisfies weak-$L^r$ inequality. Moreover, we have 
		\begin{equation}\label{h38}
		\|M_r\|_{L^r\to L^{r,\infty }}\le \ZK^{1/r}.
		\end{equation} 
	\end{theorem}
	\begin{proof}[Proof of \trm{T1-1}] 
		Denote 
		\begin{equation*}
		E=\{ x\in X:\, M_r f(x)>\lambda\}.
		\end{equation*}
		Note that $E$ can be non-measurable. For any $x\in  E $ there exists a ball $B(x)\subset X$ such that
		\begin{equation*}
		x\in B(x),\quad \frac{1}{\mu(B(x))}\int_{B(x)}|f|^r>\lambda^r.
		\end{equation*}
		We have $E=\cup_{x\in E} B(x)$. Given $B\in\ZB$ consider the collection of balls $\{B(x):\, x\in E\cap B\}$. Applying \lem{L1-1}, we find a sequence of pairwise disjoint balls $\{B_k\}$ taken from this collection such that 
		\begin{equation*}
		E \cap B\subset \bigcup_kB_k^{[1]}=Q(B).
		\end{equation*}
		We have $Q(B)$ is measurable and
		\begin{align}
		\mu(Q(B))	&\le \sum_k \mu(B^{[1]}_k)\label{a47}\\
		&\le \ZK \sum_k \mu(B_k)\nonumber\\
		&\le \frac{\ZK}{\lambda^r}\sum_k \int_{B_k}|f(t)|^rdt\nonumber\\
		&\le \frac{\ZK}{\lambda^r}\int_{X }|f(t)|^rdt.\nonumber
		\end{align}
		According to \lem {L1} there is a sequence of balls $G_k$ such that
		\begin{equation*}
		X=\bigcup_{n\ge 1}\bigcap_{k\ge n}G_k,
		\end{equation*}
		and so we get
		\begin{equation*}
		E\subset \bigcup_{n\ge 1}\bigcap_{k\ge n}Q(G_k).
		\end{equation*}
		Hence we obtain
		\begin{equation*}
		\mu^*(E)=\mu \left(\bigcup_{n\ge 1}\bigcap_{k\ge n}Q(G_k)\right)\le \frac{\ZK}{\lambda^r}\int_{X }|f(t)|^rdt
		\end{equation*}
		and so \e {h38}. 
	\end{proof}

	\begin{theorem}\label{T5-1}
		If a subadditive operator $T$ satisfies T1)-condition and the weak-$L^r$ inequality, then $T^*$ satisfies weak-$L^r$ inequality too. Moreover, we have
		\begin{equation*}
		\|T^*\|_{L^r\to L^{r,\infty}}\lesssim \ZL_1+\|T\|_{L^r\to L^{r,\infty}}.
		\end{equation*}
	\end{theorem}
	\begin{proof}
		Given $\lambda>0$ consider the set
		\begin{equation*}
		E =\{x\in X:\,T^*f(x)>\lambda \},
		\end{equation*}
		which can be non-measurable. For any $x\in E$ there is a ball $B(x)\in\ZB$ such that 
		\begin{equation}\label{h27}
		x\in B(x),\quad |T(f\cdot \ZI_{X\setminus B^{[1]}(x)})(x)|>\lambda.
		\end{equation}
		One can check that $E=\cup_{x\in E}B(x)$. Given ball $B$ apply \lem {L1-1}, we find a sequence $x_k\in E$ such that the balls $\{B_k=B(x_k)\}$ are pairwise disjoint and 
		\begin{equation}\label{h30}
		E\cap B \subset \bigcup_kB^{[1]}_k=Q(B).
		\end{equation}
		Since $T$ satisfies T1)-condition, we  have 
		\begin{equation*}
		|T(f\cdot \ZI_{X\setminus B^{[1]}_k})(x_k)-T(f\cdot \ZI_{X\setminus B^{[1]}_k})(x)|\le \ZL_1\cdot \langle f \rangle^*_{B_k,r},\quad x\in B_k. 
		\end{equation*}
		Thus, one can easily conclude from \e {h27} that
		\begin{align}
		|T(f\cdot \ZI_{X\setminus B^{[1]}_k})(x)|&\ge |T(f\cdot \ZI_{X\setminus B^{[1]}_k})(x_k)|\label{h29}\\
		&\qquad -|T(f\cdot \ZI_{X\setminus B^{[1]}_k})(x_k)-T(f\cdot \ZI_{X\setminus B^{[1]}_k})(x)|\nonumber \\
		&\ge \lambda-\ZL_1\cdot\langle f \rangle^*_{B_k,r},\quad x\in B_k.\nonumber
		\end{align}
		Given $\beta>0$ define 
		\begin{equation}\label{h28}
		\tilde B_k= \{x\in B_k:\, |T(f\cdot \ZI_{B_k^{[1]}})(x)|<\beta\cdot \langle f \rangle^*_{B_k,r}\}.
		\end{equation} 
		Using weak-$L^r$ inequality of the operator $T$, the measure of the complement of $\tilde B_k$ is estimated by
		\begin{align*}
		\mu^*(\tilde B_k^c)&\le \frac{ \|T\|_{L^r\to L^{r,\infty}}^r}{(\beta \cdot \langle f \rangle^*_{B_k,r})^r}\cdot \int_{B_k^{(1)}}|f|^r\le \left(\frac{\|T\|_{L^r\to L^{r,\infty}}}{\beta}\right)^r\mu(B_k^{(1)}) \\
		&\lesssim \left(\frac{\|T\|_{L^r\to L^{r,\infty}}}{\beta}\right)^r\mu(B_k) 
		\end{align*}
		and so for an appropriate constant $\beta\sim \|T\|_{L^r\to L^{r,\infty}}$ we have
		\begin{equation*}
		\mu^*(\tilde B_k)\ge\mu(B_k)-\mu^*(B_k\setminus \tilde B_k) \ge\mu(B_k)-\mu^*(\tilde B_k^c) \ge \frac{1}{2}\mu(B_k).
		\end{equation*}
		If
		\begin{equation*}
		x\in \tilde B_k\setminus \{M_rf(x)>\delta\lambda\},
		\end{equation*}
		then, using subadditivity of $T$ together with relations \e {h28}, \e {h29}, we obtain
		\begin{align*}
		|Tf(x)|&\ge |T(f\cdot \ZI_{X\setminus B_k^{[1]}})(x)|-|T(f\cdot \ZI_{B_k^{[1]}})(x)|\\
		&\ge \lambda -\ZL_1\cdot \langle f \rangle^*_{B_k,r}-\beta\cdot \langle f \rangle^*_{B_k,r}\\
		&\ge \lambda -(\ZL_1+\beta)\cdot M_rf(x)\\
		&\ge \lambda -(\ZL_1+\beta)\delta\lambda\\
		&\ge \lambda/2,   	
		\end{align*} 
		where the last inequality can be satisfied for 
		\begin{equation}\label{z43}
		\delta= 1/2(\ZL_1+\beta)\sim (\ZL_1+\|T\|_{L^r\to L^{r,\infty}})^{-1}.
		\end{equation}
		Hence we conclude
		\begin{equation}\label{h32}
		\bigcup_k\tilde B_k\subset \{M_rf(x)>\delta\lambda\}\bigcup \{|Tf(x)|>\lambda/2\}.
		\end{equation}
		Since the maximal function $M_r$ and the operator $T$ satisfy weak-$L^r$ bound, from \e {h30}, \e {z43} and \e {h32} we get
		\begin{align*}
		\mu(Q(B))&\le \sum_k\mu(B_k^{[1]})\\
		&\le \ZK\cdot \sum_k\mu(B_k)\\
		&\le 2\ZK\cdot  \sum_k\mu^*(\tilde B_k)\\
		&\lesssim  \mu^*\{M_rf(x)>\delta\lambda\}+\mu^*\{|Tf(x)|>\lambda/2\}\\
		&\lesssim (\ZL_1+\|T\|_{L^r\to L^{r,\infty}})^r\frac{1}{\lambda^r}\int_X|f|^r.
		\end{align*}
		The same argument used at the end of the proof of \trm {T1-1} implies
		\begin{equation*}
		\mu^*(E)\lesssim (\ZL_1+\|T\|_{L^r\to L^{r,\infty}})^r\frac{1}{\lambda^r}\int_X|f|^r,
		\end{equation*}
		and so the theorem is proved.
	\end{proof}
	\begin{definition}
		Let $T$ be a subadditive operator. Given balls $A$ and $B$ with $A\subset B$ we denote
		\begin{equation*}
		\Delta(A,B)=\Delta_T(A,B)=\sup_{x\in A,\, f\in L^r(X)}\frac{|T(f\cdot \ZI_{B^{[1]}\setminus A^{[1]}})(x)|}{\langle f\rangle_{B^{[1]},r}}.
		\end{equation*}
	\end{definition}
	Notice that T2)-condition for a subadditive operator $T$ means that for any $A\in \ZB$ there exists a ball $B$ such that $A\subsetneq B$ and $\Delta(A,B)\le \ZL_2$. 
	\begin{lemma}\label{L19}
		If $T$ is an arbitrary subadditive operator, then for any balls $A,B$ and $C$ satisfying $A\subset B\subset C$ we have
		\begin{equation}\label{a63}
		\Delta(A,B)\le\Delta(A,C).
		\end{equation}
	\end{lemma}
	\begin{proof}
		Given function $f\in L^r(X)$ denote 
		\begin{equation*}
		g(x)=f(x)\cdot \ZI_{B^{[1]}}(x).
		\end{equation*}
		Then we get the estimate
		\begin{align*}
		\sup_{x\in A}|T(f\cdot \ZI_{B^{[1]}\setminus A^{[1]}})(x)|&=\sup_{x\in A}|T(g\cdot\ZI_{C^{[1]}\setminus A^{[1]}})(x)|\\
		&\le \Delta(A,C)\cdot \langle g\rangle_{C^{[1]},r}\\
		&= \Delta(A,C)\cdot \left(\frac{1}{\mu(C^{[1]})}\int_{B^{[1]}}|f|^r\right)^{1/r}\\
		&\le \Delta(A,C)\cdot \langle f\rangle_{B^{[1]},r},
		\end{align*}
		which implies \e {a63}. 
	\end{proof}
	\begin{lemma}\label{L4}
		Let a subadditive operator $T$ satisfy T1)-condition and the weak-$L^r$ bound. Then for any balls $A$, $B$ satisfying $A\subset B$ we have
		\begin{equation}\label{a99}
		\Delta(A,B)\lesssim (\ZL_1+\|T\|_{L^r\to L^{r,\infty}})\left(\frac{\mu(B)}{\mu(A)}\right)^{1/r}.
		\end{equation}	
	\end{lemma} 
	\begin{proof}
		Applying the weak-$L^r$ estimate, we get
		\begin{equation*}
		\mu^*\left\{x\in A:\, |T(f\cdot\ZI_{B^{[1]}\setminus A^{[1]}})(x)|>\|T\|_{L^r\to L^{r,\infty}}\left(\frac{2}{\mu(A)}\int_{B^{[1]}}|f|^r\right)^{1/r}\right\}\le \frac{\mu(A)}{2}
		\end{equation*}	
		and so we find a point $x_0\in A$ such that
		\begin{align}
		|T(f\cdot\ZI_{B^{[1]}\setminus A^{[1]}})\}(x_0)|&\le \|T\|_{L^r\to L^{r,\infty}}\left(\frac{2}{\mu(A)}\int_{B^{[1]}}|f|^r\right)^{1/r}\label{f1}\\
		&\lesssim \|T\|_{L^r\to L^{r,\infty}}\cdot\left(\frac{\mu(B)}{\mu(A)}\right)^{1/r}\cdot \langle f\rangle_{B^{[1]},r}.\nonumber
		\end{align}
		According to T1)-condition, for any $x\in A$ we have 
		\begin{equation}\label{f2}
		|T(f\cdot\ZI_{B^{[1]}\setminus A^{[1]}})\}(x)-T(f\cdot\ZI_{B^{[1]}\setminus A^{[1]}})\}(x_0)|\le \ZL_1\cdot \langle f\cdot\ZI_{B^{[1]}\setminus A^{[1]}} \rangle^*_{A,r}.
		\end{equation}
		By the definition of $\langle f \rangle^*_{A,r}$ there is a ball $C\supset A$ such that
		\begin{equation*}
		\langle f\cdot\ZI_{B^{[1]}\setminus A^{[1]}} \rangle^*_{A,r}=\langle f\cdot\ZI_{B^{[1]}\setminus A^{[1]}}\rangle_{C,r}.
		\end{equation*}
		If $\mu(C)\le \mu(B^{[1]})$, then we have $C\subset B^{[2]}$ and therefore 
		\begin{align}
		\langle f\cdot\ZI_{B^{[1]}\setminus A^{[1]}}\rangle_{C,r}&\le \left(\frac{\mu(B^{[2]})}{\mu(C)}\right)^{1/r}\langle f\cdot\ZI_{B^{[1]}\setminus A^{[1]}}\rangle_{B^{[2]},r}\label{h34}\\
		&=\left(\frac{\mu(B^{[2]})}{\mu(C)}\right)^{1/r}\langle f\cdot\ZI_{B^{[1]}\setminus A^{[1]}}\rangle_{B^{[1]},r}\nonumber\\
		&\lesssim \left(\frac{\mu(B)}{\mu(A)}\right)^{1/r}\cdot \langle f\rangle_{B^{[1]},r}.\nonumber
		\end{align}
		In the case of $\mu(C)> \mu(B^{[1]})$, we have $B^{[1]}\subset C^{[1]}$. Hence we get
		\begin{align}
		\langle f\cdot\ZI_{B^{[1]}\setminus A^{[1]}}\rangle_{C,r}&\le \left(\frac{\mu(C^{[1]})}{\mu(C)}\right)^{1/r}\langle f\cdot\ZI_{B^{[1]}\setminus A^{[1]}}\rangle_{C^{[1]},r}\label{h35}\\
		&\lesssim\langle f\cdot\ZI_{B^{[1]}\setminus A^{[1]}}\rangle_{C^{[1]},r}\nonumber\\
		&\le\langle f\cdot\ZI_{B^{[1]}\setminus A^{[1]}}\rangle_{B^{[1]},r}\nonumber\\
		&\le \langle f\rangle_{B^{[1]},r}.\nonumber
		\end{align}
		The estimates \e {h34} and \e {h35} imply the inequality 
		\begin{equation*}
		\langle f\cdot\ZI_{B^{[1]}\setminus A^{[1]}} \rangle^*_{A,r}\lesssim \left(\frac{\mu(B)}{\mu(A)}\right)^{1/r}\cdot \langle f\rangle_{B^{[1]},r},
		\end{equation*}
		which together with \e {f1} and \e {f2} gives
		\begin{equation*}
		|T(f\cdot\ZI_{B^{[1]}\setminus A^{[1]}})\}(x)|\lesssim (\ZL_1+\|T\|_{L^r\to L^{r,\infty}})\left(\frac{\mu(B)}{\mu(A)}\right)^{1/r}\langle f\rangle_{B^{[1]},r},\quad x\in A.
		\end{equation*}
		The last inequality completes the proof of lemma.
	\end{proof} 
	\begin{lemma}\label{L7}
		If a subadditive operator $T$ satisfies T1)-condition and the weak-$L^r$ bound, then for any balls $A,B$ and $C$ satisfying $A\subset B\subset C$ we have
		\begin{equation}\label{h51}
		\Delta(A,C)\lesssim\left(\frac{\mu(C)}{\mu(B)}\right)^{1/r}\cdot (\ZL_1+\|T\|_{L^r\to L^{r,\infty}}+\Delta(A,B)).
		\end{equation}
	\end{lemma}
	\begin{proof}
		According to \lem {L4} we have 
		\begin{equation*}
		\Delta(B,C)\lesssim (\ZL_1+\|T\|_{L^r\to L^{r,\infty}})\left(\frac{\mu(C)}{\mu(B)}\right)^{1/r}.
		\end{equation*}
		Thus, taking $f\in L^r(X)$ and $x\in A$, we obtain
		\begin{align*}
		|T(f\cdot &\ZI_{C^{[1]}\setminus A^{[1]}})(x)|\\
		&\le |T(f\cdot \ZI_{C^{[1]}\setminus B^{[1]}})(x)|+|T(f\cdot \ZI_{B^{[1]}\setminus A^{[1]}})(x)|\\
		&\le\Delta(B,C)\langle f\rangle_{C^{[1]},r}+\Delta(A,B) \langle f\rangle_{B^{[1]},r}\\
		&\lesssim (\ZL_1+\|T\|_{L^r\to L^{r,\infty}})\left(\frac{\mu(C)}{\mu(B)}\right)^{1/r}\langle f\rangle_{C^{[1]},r}+\Delta(A,B) \langle f\rangle_{B^{[1]},r}\\
		&\le (\ZL_1+\|T\|_{L^r\to L^{r,\infty}})\left(\frac{\mu(C)}{\mu(B)}\right)^{1/r}\langle f\rangle_{C^{[1]},r}\\
		&\qquad +\Delta(A,B)\left(\frac{\mu(C^{[1]})}{\mu(B^{[1]})}\right)^{1/r} \langle f\rangle_{C^{[1]},r}\\
		&\lesssim \langle f\rangle_{C^{[1]},r}\left(\frac{\mu(C)}{\mu(B)}\right)^{1/r}\cdot(\ZL_1+\|T\|_{L^r\to L^{r,\infty}}+\Delta(A,B))
		\end{align*}
		and so we get \e {h51}.
	\end{proof}
	Inequality \e {a33} and \lem {L7} immediately yield
	\begin{lemma}\label{L-2}
		Let an operator $T\in \BO_\ZB$ and satisfy the weak-$L^r$ bound. Then for any balls $A,B\in\ZB$ with $A\subset B$ we have 
		\begin{equation*}
		\Delta(A,B^{[n]})\lesssim \ZK^{n/r}(\ZL_1+\|T\|_{L^r\to L^{r,\infty}}+\Delta(A,B)).
		\end{equation*}
	\end{lemma}
	
	\begin{theorem}\label{T6}
		If an operator $T\in \BO_\ZB$ satisfies weak-$L^r$ estimate, then the operator $T^*\in \BO_\ZB$ and satisfies weak-$L^r$ inequality. Moreover, we have
		\begin{align*}
		&\ZL_1(T^*)\lesssim \ZL_1(T)+\|T\|_{L^r\to L^{r,\infty}},\\
		&\ZL_2(T^*)=\ZL_2(T). 
		\end{align*}
	\end{theorem}
	\begin{proof}
		One can easily check that for any balls $A,B$ satisfying $A\subset B$ we have the equality
		\begin{equation*}
		\Delta_{T^*}(A,B)=\Delta_{T}(A,B).
		\end{equation*}
		If balls $A$ and $B$ satisfy T2)-condition for the operator $T$, then the same conditions will hold also for $T^*$ with $\ZL_2(T^*)=\ZL_2(T)$. 
		To prove T1)-condition, let $B\in \ZB$ and $f\in L^r(X)$ satisfy 
		\begin{equation}\label{a67}
		\supp f\in X\setminus B^{[1]}.
		\end{equation}
		Take arbitrary points $x,x'\in B$ and estimate $|T^*f(x)- T^*f(x')|$. If $T^*f(x)= T^*f(x')$, then the estimation is trivial. So we can suppose that $T^*f(x)> T^*f(x')$. Using the definition of $T^*f(x)$, we find a ball $A\in\ZB$ with $x\in A$ such that
		\begin{equation}\label{h37}
		\frac{T^*f(x)+T^*f(x')}{2}< |T(f\cdot \ZI_{X\setminus A^{[1]}})(x)|.
		\end{equation} 
		Denote 
		\begin{align}\label {a60}
		A'=\left\{
		\begin{array}{lc}
		B&\hbox{ if } \mu(A)<\mu(B),\\
		A^{[1]}&\hbox { if } \mu(A)\ge \mu(B).\\
		\end{array}
		\right.
		\end{align}
		Since $x,x'\in B$, from \lem {L4} and \e {a67} it follows that
		\begin{align}
		|T(f\cdot \ZI_{B^{[2]}\setminus A^{[1]}})(x)|&=|T(f\cdot \ZI_{B^{[2]}\setminus A^{[1]}}\cdot \ZI_{B^{[2]}\setminus B^{[1]}}))(x)|\label{a81}\\
		&\le \Delta(B,B^{[1]}) \langle f\cdot \ZI_{B^{[2]}\setminus A^{[1]}}\rangle_{B^{[2]},r}\nonumber\\
		&\le \Delta(B,B^{[1]}) \langle f\rangle_{B^{[2]},r}\nonumber\\
		&\lesssim \left(\frac{\mu(B^{[1]})}{\mu(B)}\right)^{1/r}(\ZL_1(T)+\|T\|_{L^r\to L^{r,\infty}}) \langle f \rangle^*_{B,r},\nonumber\\
		&\lesssim (\ZL_1(T)+\|T\|_{L^r\to L^{r,\infty}}) \langle f \rangle^*_{B,r},\nonumber
		\end{align}
		and similarly
		\begin{equation}\label{a82}
		|T(f\cdot \ZI_{B^{[2]}\setminus A'^{[1]}}(x')|\lesssim (\ZL_1(T)+\|T\|_{L^r\to L^{r,\infty}}) \langle f \rangle^*_{B,r}.
		\end{equation}
		One can easily check that from \e {a60} it follows that $B\subset A'$ and so $x'\in A'$. This implies 
		\begin{equation*}
		T^*f(x')\ge |T(f\cdot \ZI_{X\setminus A'^{[1]}})(x')|,
		\end{equation*}
		which together with \e {h37} yields 
		\begin{align}
		|T^*f(x)-T^*f(x')|&=T^*f(x)-T^*f(x')\label{a6}\\
		&\le 2|T(f\cdot \ZI_{X\setminus A^{[1]}})(x)|-2T^*f(x')\nonumber\\
		&\le 2\bigg(|T(f\cdot \ZI_{X\setminus A^{[1]}})(x)|-|T(f\cdot \ZI_{X\setminus A'^{[1]}})(x')|\bigg).\nonumber
		\end{align}
		In the case $\mu(A)<\mu(B)$ we get $A\subset B^{[1]}$ and therefore 
		\begin{equation*}
		A'^{[1]}=B^{[1]}\subset B^{[2]},\quad A^{[1]}\subset B^{[2]}.
		\end{equation*}
		Thus, applying T1)-condition for $T$, from \e {a81}, \e {a82}  and \e {a6} we conclude
		\begin{align*}
		|T^*f(x)-T^*f(x')|&\le 2\bigg(|T(f\cdot \ZI_{X\setminus B^{[2]}})(x)|+|T(f\cdot \ZI_{B^{[2]}\setminus A^{[1]}})(x)|\\
		&\qquad -|T(f\cdot \ZI_{X\setminus B^{[2]}})(x')|+|T(f\cdot \ZI_{B^{[2]}\setminus 	A'^{[1]}})(x')|\bigg)\\
		&\lesssim |T(f\cdot \ZI_{X\setminus B^{[2]}})(x)-T(f\cdot \ZI_{X\setminus B^{[2]}})(x')|\\
		&\qquad +(\ZL_1(T)+\|T\|_{L^r\to L^{r,\infty}})\langle f \rangle^*_{B,r} \\
		&\lesssim (\ZL_1(T)+\|T\|_{L^r\to L^{r,\infty}}) \langle f \rangle^*_{B,r}.
		\end{align*}
		If $\mu(A)\ge \mu(B)$, then by \e {a60} we have $A'=A^{[1]}$ and so $x,x'\in B\subset A^{[1]}$, $A'^{[1]}=A^{[2]}$. Thus, applying \lem {L4} and \e {a6}, we get
		\begin{align*}
		|T^*f(x)-T^*f(x')|&\le 2\bigg(|T(f\cdot \ZI_{X\setminus A^{[1]}})(x)|-|T(f\cdot \ZI_{X\setminus A^{[2]}})(x')|\bigg)\\
		&\le 2\bigg(|T(f\cdot \ZI_{X\setminus A^{[2]}})(x)|-|T(f\cdot \ZI_{X\setminus A^{[2]}})(x')|\bigg)\\
		&\qquad +2|T(f\cdot \ZI_{A^{[2]}\setminus A^{[1]}})(x)|\\
		&\lesssim \ZL_1(T)\langle f \rangle^*_{A^{[1]},r}+(\ZL_1(T)+\|T\|_{L^r\to L^{r,\infty}})\langle f\rangle_{A^{[2]},r}\\
		&\le(\ZL_1(T)+\|T\|_{L^r\to L^{r,\infty}})\langle f \rangle^*_{B,r},
		\end{align*}
		and finally, T1)-condition. Weak-$L^r$ bound of $T^*$ follows from \trm {T5-1}.
	\end{proof}
	We say that a ball-basis $\ZB$ in a measure space satisfies the doubling condition if there is a constant $\eta>1$ such that for any ball $A\in \ZB$, $A^{[1]}\neq X$, one can find a ball $B$ satisfying
	\begin{equation}\label{h73}
	A\subsetneq B,\quad \mu( B)\le\eta  \cdot \mu(A).
	\end{equation}
	\begin{theorem}\label{T4}
		Let the ball-basis $\ZB$ satisfy the doubling condition. If a subadditive operator $T$ satisfies T1)-condition and the weak-$L^r$ bound, then $T\in\BO_\ZB$. Moreover, we have 
		\begin{equation*}
		\ZL_2(T)\lesssim \eta^{1/r}(\ZL_1+\|T\|_{L^r\to L^{r,\infty}}), 
		\end{equation*}
		where $\eta$ is the doubling constant from \e {h73}.
	\end{theorem} 
	\begin{proof}
		We need to check T2)-condition. If balls $A$ and $B$ satisfy conditions \e {h73}, then applying \lem {L4}, we get
		\begin{equation*}
		\Delta(A,B)\lesssim (\ZL_1+\|T\|_{L^r\to L^{r,\infty}})\left(\frac{\mu( B)}{\mu(A)}\right)^{1/r}\le \eta^{1/r} (\ZL_1+\|T\|_{L^r\to L^{r,\infty}}).
		\end{equation*}
		Thus we get $\ZL_2(T)\lesssim \eta^{1/r}(\ZL_1+\|T\|_{L^r\to L^{r,\infty}})<\infty$.
	\end{proof}
	\section{Proof of main theorems}
	\begin{lemma}\label{L21}
		If a subadditive operator $T$ satisfies T1)-condition and the weak-$L^r$ bound, then for any $B\in \ZB$ there exists a ball $B'$ such that
		\begin{align}
		&B^{[2]}\subset B',\label{z21}\\
		&\Delta(B^{[2]},B')\lesssim   \ZL_1+\|T\|_{L^r\to L^{r,\infty}},\label{z22}
		\end{align}
		and we either have 
		\begin{equation}\label{z23}
		B'^{[1]}=B'\text { or }\mu(B')\ge 2\mu(B).
		\end{equation}
	\end{lemma}
	\begin{proof}
		Letting  
		\begin{equation*}
		\zA=\{A\in\ZB:\,A\cap B\neq\varnothing\},
		\end{equation*}
		we denote
		\begin{align}
		&a=\sup_{A\in\zA:\, \mu(B)\le \mu(A)\le 2\mu(B)}\mu(A)\le 2\mu(B),\label{z44}\\
		&b=\inf_{A\in\zA:\, \mu(A)> 2\mu(B)}\mu(A)\ge 2\mu(B).\label{z45}
		\end{align}
		Observe that there is no ball $A\in\zA$ with $a<\mu(A)<b$ and there exist balls $G_1,G_2\in\zA$ such that
		\begin{align*}
		&\frac{a}{2}<\mu(G_1)\le a,\\
		&b\le \mu(G_2)< 2b.
		\end{align*}
		If $b\le \ZK^2a$, then we define $B'=G_2^{[3]}$. 	Since $B\cap G_2\neq\varnothing $ and $\mu(B)\le a\le b\le \mu(G_2)$, we get $B\subset G_2^{[1]}$ and therefore $B^{[2]}\subset G_2^{[3]}=B'$. Thus we get \e {z21}. Taking into account the inequalities
		\begin{align*}
		&\mu(B')=\mu(G_2^{[3]})\le \ZK^3\mu(G_2)\le \ZK^3\cdot 2b\le 2a\ZK^5,\\
		&\mu(B^{[2]})\ge \mu(B)\ge \frac{a}{2},
		\end{align*}
		from \lem{L4} it follows that
		\begin{equation*}
		\Delta(B^{[2]},B')\lesssim \left(\frac{\mu(B')}{\mu(B^{[2]})}\right)^{1/r}(\ZL_1+\|T\|_{L^r\to L^{r,\infty}})\lesssim  \ZL_1+\|T\|_{L^r\to L^{r,\infty}}.
		\end{equation*} 
		Hence we obtain \e {z22}. Then by \e {z45} we have
		\begin{equation}\label{z33}
		\mu(B')\ge \mu(G_2)\ge b\ge 2\mu(B)
		\end{equation}
		and so we get the second relation in \e {z23}. Now suppose we have $b> \ZK^2a$. 
		Define $B'=G_1^{[1]}\in \zA$. We have
		\begin{equation*}
		\mu(B'^{[1]})\le \ZK^2\mu(G_1)\le a\ZK^2<b.
		\end{equation*}
		Since there is no ball from $\zA$ with a measure in the interval $(a,b)$, we get $\mu(B'^{[1]})\le a$. Thus we get
		\begin{equation*}
		B'^{[1]}\cap G_1\neq\varnothing,\quad \mu(B'^{[1]})\le 2\mu(G_1).
		\end{equation*}
		These relations imply $B'^{[1]}\subset G_1^{[1]}=B'$, that means $B'^{[1]}=B'$ and we get the first relation in \e {z23}. 
		On the other hand by \e {z44} we have
		\begin{equation*}
		\mu(B)\le a\le 2\mu(G_1)\le 2\mu(B')\le 2a\le 4\mu(B). 
		\end{equation*}
		That means $B\subset B'^{[1]}=B'$ and therefore $B^{[2]}\subset B'^{[2]}=B'$. Hence we obtain \e {z21}. Using  \lem {L4} and \e {z33}, we get
		\begin{equation*}
		\Delta(B^{[2]}, B')\lesssim \left(\frac{\mu(B')}{\mu(B^{[2]})}\right)^{1/r}(\ZL_1+\|T\|_{L^r\to L^{r,\infty}})\lesssim \ZL_1+\|T\|_{L^r\to L^{r,\infty}},
		\end{equation*}
		that gives \e {z22}. Lemma is proved.
	\end{proof}
	\begin{lemma}\label{L22}
		Let $T\in\BO_\ZB$ satisfy weak-$L^r$ estimate and for a ball $B\in\ZB$ we have $B^{[1]}=B$. Then there exists a ball $B'$ satisfying \e {z21} and 
		\begin{align}
		&\Delta(B^{[2]},B')\le   \ZL_2,\label{z27}\\
		&\mu(B')>2\mu(B).\label{z46}
		\end{align}
	\end{lemma}
	\begin{proof}
		Applying T2)-condition, we find a ball $B'\supsetneq B$ such that
		\begin{equation}\label{z26}
		\Delta(B,B')\le \ZL_2.
		\end{equation}
		From $B^{[1]}=B$ we get $B^{[2]}=B\subset B'$ that is \e {z21}. Then, applying  \e {z26}, we get \e {z27}. 
		Observe that \e {z46} holds. Indeed, otherwise by two balls relation we will have $B'\subset  B^{[1]}=B$, which contradicts the condition $B'\supsetneq B$.
	\end{proof}
	\begin{lemma}\label{L8}
		If $T\in\BO_\ZB$ satisfies weak-$L^r$ estimate, then for any $B\in \ZB$ there exists a sequence of balls $B=B_0, B_1, B_2,\ldots $ such that  
		\begin{align}
		&\bigcup_kB_k=X,\label{h9} \\
		&B_{k-1}^{[2]}\subset B_{k},\quad k\ge 1,\label{h46}\\
		&\sup_{k\ge 1}\Delta(B_{k-1}^{[2]},B_{k})\lesssim   \ZL_1+\ZL_2+\|T\|_{L^r\to L^{r,\infty}},\quad k\ge 1.\label{h41}
		\end{align}
	\end{lemma}
	\begin{proof}
		We construct a sequence of balls $\{B_k\}$ satisfying \e {h46}, \e {h41} and the relation
		\begin{equation}\label{z32}
		\mu(B_{k})>2\mu(B_{k-2}),\quad k\ge 2.
		\end{equation}
		We do it by induction. Take $B_0=B$ and suppose we have already defined $B_k$ for $k=0,1,\ldots l$ satisfying the conditions \e {h46}, \e {h41} and \e {z32} for $k\le l$. If $B_l^{[1]}=B_l$, then applying \lem {L22}, we get a ball $B_{l+1}=B'$ such that
		\begin{align}
		&B_l^{[2]}\subset B_{l+1},\label{z28}\\
		&\Delta(B_l^{[2]},B_{l+1})\le \ZL_2,\nonumber\\
		&\mu(B_{l+1})\ge 2\mu(B_l).\nonumber
		\end{align}
		In the case of $B_l^{[1]}\neq B_l$ we use \lem {L21}. At this time the ball $B_{l+1}=B'$ will satisfy \e {z28} and the conditions
		\begin{align}
		&\Delta(B_l^{[2]},B_{l+1})\lesssim \ZL_1+\|T\|_{L^r\to L^{r,\infty}},\nonumber\\
		&\mu(B_{l+1})\ge 2\mu(B_l).\label{z30}
		\end{align}
		Applying this process, we get a sequence satisfying the conditions \e {h46} and \e {h41}. Besides, one can easily observe that at least for one of any two consecutive integers $k$ we will have $\mu(B_{k+1})\ge 2\mu(B_k)$. So the condition \e {z32} will also be satisfied.  
	\end{proof}
	
	\begin{lemma}\label{L5-1}
		Let $T$ be a $\BO_\ZB$ operator satisfying weak-$L^r$ inequality. If 
		\begin{equation}\label{a80}
		\lambda\ge 3\ZK^4,
		\end{equation}
		then for every measurable set $F\subset X$ and a ball $A\in \ZB$ with 
		\begin{equation}\label{a83}
		F\cap A\neq\varnothing,\quad \mu(F)\le\mu(A)/\lambda,
		\end{equation}
		there exists a family of balls $\ZG\subset \ZB$ satisfying the conditions
		\begin{align}
		&F\cap A^{[1]}\cap G \neq\varnothing\text { if } G\in\ZG,\label{a18}\\
		&F\cap A^{[1]}\subset \bigcup_{G\in\ZG}  G\text { a.s. }\label{a19}\\
		&\mu\left(\bigcup_{G\in\ZG} G^{[1]}\right)\lesssim \frac{\mu(A)}{\lambda}.\label{a20}
		\end{align}
		Besides, for each  $G\in\ZG$ there is a ball $\tilde G$ (not necessarily in $\ZG$) such that 
		\begin{align}
		&\tilde G\not\subset F.\label{a22}\\
		&G^{[2]}\subset \tilde G\subset A^{[1]},\label{a21}\\
		&\Delta(G^{[2]},\tilde G)\lesssim  \ZL_1+\ZL_2+\|T\|_{L^r\to L^{r,\infty}}.\label{a61}
		\end{align}
	\end{lemma}
	\begin{proof}
		Applying \lem {L3} for the set $E=F\cap A^{[1]}$ we find a family of balls $\ZP\subset \ZB$ such that
		\begin{align}
		&F\cap A^{[1]}\cap B\neq \varnothing, \quad B\in \ZP,\label{a8}\\
		&F\cap A^{[1]}\subset\bigcup_{B\in\ZP} B\text { a.s. }\label{a9}\\
		&\sum_{B\in\ZP} \mu(B)<2\ZK\mu(F\cap A^{[1]}). \label{a10}
		\end{align}
		Take an arbitrary element $B\in\ZP$. Applying \lem {L8}, we find a sequence of balls $B_k\in \ZB$, $k=0,1,2,\ldots$, $B=B_0$, with conditions \e {h9}-\e {h41}. To $B\in\ZP$ we can attach a ball $G=B_{m}$, where $m\ge 0$ is the least index satisfying the relation
		\begin{equation}\label{a16}
		B_{m+1}^{[1]}\not\subset F.
		\end{equation}
		The collection of all such $G$ defines the family $\ZG$. 
		If $G\in \ZG$ is generated by $B\in\ZP$, then combining the relation 
		\begin{equation}\label{a25}
		B\subset B_0^{[2]}\subset B_m=G
		\end{equation}
		with \e {a8} we obtain \e {a18}. From \e {a25} we also get 
		\begin{equation*}
		\bigcup_{G\in\ZG} G\supset \bigcup_{B\in\ZP} B,
		\end{equation*}
		which together with \e {a9} implies \e {a19}. Then, according to the definition of the integer $m$ (see \e {a16}) we have either $G^{[1]}\subset F$ or $G\in \ZP$ (the second relation holds only if $m=0$). This remark together with \e {a10} implies 
		\begin{align}
		\mu\left(\bigcup_{G\in\ZG}G^{[1]}\right)&\le \mu(F)+\mu\left(\bigcup_{B\in\ZP} B^{[1]}\right)\label{a23}\\
		&\le \mu(F)+\sum_{B\in\ZP} \mu\left(B^{[1]}\right)\nonumber\\
		&\le \mu(F)+\ZK\sum_{B\in\ZP} \mu(B),\nonumber\\
		&\le \frac{(2\ZK^2+1)\mu(A^{[1]})}{\lambda},\nonumber\\
		&\le \frac{3\ZK^2\mu(A)}{\lambda},\nonumber
		\end{align}
		and we get \e {a20}. Now define 
		\begin{equation}\label{a68}
		\tilde G=\left\{
		\begin{array}{lcl}
		A^{[1]}&\hbox{ if }& \mu(B_{m+1}^{[1]})>\mu(A),\\
		B_{m+1}^{[1]}&\hbox{ if }&\mu(B_{m+1}^{[1]})\le \mu(A).
		\end{array}
		\right.
		\end{equation}
		According to \e {a83}, we have $\mu(A^{[1]})>\mu(F)$ and so $A^{[1]}\not\subset F$. This together with \e {a16} and \e {a68} implies \e {a22}. To check condition \e {a21} notice that form \e {a80} and \e {a23} it follows that
		\begin{equation*}
		\mu(G^{[2]})\le \ZK^2\cdot \mu(G)\le  \ZK^2 \cdot \frac{3\ZK^2\mu(A)}{\lambda}\le \mu(A).
		\end{equation*}
		Thus, since $G\cap A\neq \varnothing$, we conclude
		\begin{equation}\label{a24}
		G^{[2]}\subset A^{[1]}.
		\end{equation}
		If $\mu(B_{m+1}^{[1]})>\mu(A)$, then by \e {a68} we have $\tilde G=A^{[1]}$ and using \e {a24} we get \e {a21}. If $\mu(B_{m+1}^{[1]})\le \mu(A)$, then since $B_{m+1}^{[1]}\cap A\neq\varnothing$, we have $B_{m+1}^{[1]}\subset A^{[1]}$. Hence from \e {h46} and \e {a68} we obtain 
		\begin{equation*}
		G^{[2]}=B_{m}^{[2]}\subset  B_{m+1}^{[1]}=\tilde G\subset A^{[1]}
		\end{equation*}
		and so \e {a21}. To prove \e {a61} first we suppose that $\mu(B_{m+1}^{[1]})\le \mu(A)$ and so $\tilde G=B_{m+1}^{[1]}$. Applying \lem {L-2} and \e {h41}, we get
		\begin{align*}
		\Delta(G^{[2]},\tilde G)&=\Delta(B_m^{[2]},B_{m+1}^{[1]})\\
		&\lesssim \left(\frac{\mu(B_{m+1}^{[1]})}{\mu(B_{m+1})}\right)^{1/r}(\ZL_1+\|T\|_{L^r\to L^{r,\infty}}+\Delta(B_m^{[2]},B_{m+1}))\\
		&\lesssim  \ZL_1+\ZL_2+\|T\|_{L^r\to L^{r,\infty}}.
		\end{align*}
		In the case $\mu(B_{m+1}^{[2]})> \mu(A^{[1]})$ we have $\tilde G=A^{[1]}\subset B_{m+1}^{[3]}$ and therefore $\tilde G^{[1]}\subset B_{m+1}^{[4]}$. Once again applying \lem {L-2}, we obtain
		\begin{align*}
		\Delta(G^{[2]},\tilde G)&\le \Delta(B_m^{[2]},B_{m+1}^{[3]})\\
		&\lesssim \left(\frac{\mu(B_{m+1}^{[2]})}{\mu(B_{m+1})}\right)^{1/r}( \ZL_1+\ZL_2+\|T\|_{L^r\to L^{r,\infty}}+\Delta(B_m^{[2]},B_{m+1}))\\
		&\lesssim  \ZL_1+\ZL_2+\|T\|_{L^r\to L^{r,\infty}},
		\end{align*}
		which completes the proof of lemma.
	\end{proof}
	\begin{definition}
		We say a set of balls $\zA$ is a family-tree if 
		\begin{enumerate}
			\item [F1)]there is an element $A_0\in \zA$ called grandparent of $\zA$, 
			\item [F2)]to each $A\in \zA$ except the grandparent $A_0$ a unique parent $\pr(A)\in \zA$ is attached,
			\item [F3)]for each $A\in \zA$, $A\neq A_0$ we have $A_0=\pr^n(A)=\pr(\pr(\ldots \pr(A)\ldots))$ for some $n\in \ZN$.
		\end{enumerate}
	\end{definition}
	Given ball $A\in \zA$ we denote
	\begin{align*}
	&\Ch_n(A)=\{B\in\zA:\, \pr^n(B)=A\},\quad n=1,2,\ldots \\
	&\Gen(A)=\bigcup_{n=1}^\infty \Ch_n(A),
	\end{align*}
	where $M$ is the maximal operator \e {1-1}. The family $\Ch(A)=\Ch_{1}(A)$ is said to be the children of $A$ and $\Gen(A)$ is the generation of $A$. 
	
	The notation $n\ll m$ ($n\gg m$) for two integers $n,m$ denotes $n<m-1$ ($n> m+1$) and $n\asymp m$ stands for the condition $|m-n|\le 1$.
	
	\begin{proof}[Proof of \trm {T1}]
		Let an operator $T\in \BO_\ZB$ satisfy weak-$L^r$ inequality. Define 
		\begin{align*}
		\Gamma f(x)=\max\left\{|Tf(x)|,T^{*}f(x),\ZL\cdot M_rf(x)\right\},
		\end{align*}
		where
		\begin{equation*}
		\ZL=\ZL_1+\ZL_2+\|T\|_{L^r\to L^{r,\infty}}.
		\end{equation*}
		Applying \trm {T1-1} and \trm {T5-1}, we conclude that the operator $\Gamma$ satisfy weak-$L^r$ estimate and besides
		\begin{equation}\label{h62}
		\|\Gamma\|_{L^r\to L^{r,\infty}}\lesssim \ZL_1+\ZL_2+\|T\|_{L^r\to L^{r,\infty}}.
		\end{equation}
		Denote 
		\begin{equation}\label{h57}
		T^{**}f(x)=\sup_{B\in \ZB,\, x\in B}|T(f\cdot \ZI_{B^{[1]}})(x)|.
		\end{equation}
		Subadditivity of $T$ implies 
		\begin{equation}\label{a5}
		T^{**}f(x)\le T^{*}f(x)+|Tf(x)|\le 2\Gamma f(x),\quad x\in X.
		\end{equation} 
		
		Let $f\in L^r(X)$ and $B$ be the ball from the statement of theorem. Clearly we can chose a ball $A_0$ such that 
		\begin{equation}\label{h104}
		B^{[1]}\subset A_0,\quad \int_{A_0}|f|>\frac{\|f\|_1}{2}.
		\end{equation}
		Let $\lambda>0$ satisfy \e {a80}.	We shall construct a family-tree $\zA\subset \ZB$ with the grandparent $A_0$ such that 
		\begin{enumerate}
			\item [1)] If $G\in \Ch(A)$, then $A^{[1]}\cap G\neq \varnothing$.
			\item [2)] We have
			\begin{equation}\label{a72}
			\mu\left(\bigcup_{G\in \Ch(A)}G^{[1]}\right)\lesssim \frac{\mu(A)}{\lambda}.
			\end{equation}
			\item [3)] If $G\in \Ch(A)$, then there exist a ball $\tilde G$ and a point $\xi\in \tilde G$ such that
			\begin{align}
			&G^{[2]}\subset\tilde G \subset A^{[1]},\label{a12}\\
			&\Gamma f_A(\xi)\lesssim  \ZL\lambda \cdot \langle f\rangle_{A^{[3]},r},\label{a13}\\
			&|T(f\cdot \ZI_{\tilde G^{[1]}\setminus  G^{[3]}})(x)|\lesssim \ZL\lambda\cdot \langle f\rangle_{A^{[3]},r},\, x\in G^{[2]},\label{a14}	
			\end{align}
			where
			\begin{equation}\label{h106}		
			f_A=\left\{\begin{array}{lcl}
			&f\cdot \ZI_{A^{[3]}}\hbox{ if }& A\neq A_0,\\
			&f\hbox{ if }&A=A_0 .
			\end{array}
			\right.
			\end{equation}
			\item [4)] We have
			\begin{align}
			&|\Gamma f_A(x)|\lesssim  \ZL\lambda\cdot \langle f\rangle_{A^{[3]},r},\label{d23}\\
			&\qquad\qquad \text {for almost all }  x\in A^{[1]}\setminus \bigcup_{G\in \Ch(A)} G,\text { if }A\neq A_0.\nonumber
			\end{align}
		\end{enumerate}
		The elements of $\zA$ will be determined inductively by an increasing order of generations levels. The first element of $\zA$ is $A_0$. Then we suppose inductively that we have already defined all the members of 
		\begin{equation*}
		\bigcup_{1\le n\le l} \Ch_n(A)
		\end{equation*}
		such that any member 
		\begin{equation*}
		A\in \bigcup_{1\le n\le l-1} \Ch_n(A)
		\end{equation*}
		satisfies the conditions 1)-4). To define the members of $\Ch_{l+1}(A_0)$ we take an arbitrary $A\in \Ch_{l}(A_0)$ and define the children of $A$ as follows. Take a measurable set $F=F_A$ such that
		\begin{align}
		&F\supset \{x\in X:\, \Gamma f_A(x)>\beta\cdot  \langle f\rangle_{A^{[3]},r}\},\label{a26}\\
		&\mu(F)=\mu^*\{x\in X:\, \Gamma f_A(x)>\beta\cdot  \langle f\rangle_{A^{[3]},r}\}.
		\end{align}
		Suppose that $F\neq\varnothing$. Using \e {h62}, for any $A$ (include the case $A=A_0$) we will have
		\begin{align*}
		\mu(F)&\le \frac{2\mu(A^{[3]})}{\beta^r}\|\Gamma\|^r_{L^r\to L^{r,\infty}}\\
		&\le \frac{2\ZK^3\cdot \mu(A)}{\beta^r}\|\Gamma\|^r_{L^r\to L^{r,\infty}}\\
		&\le \frac{\mu(A)}{\lambda},
		\end{align*}
		for a suitable constant 
		\begin{equation*}
		\beta\sim \ZL\lambda
		\end{equation*}
		since $r\ge 1$ and we have \e {h62}. Not that  in the case $A=A_0$ one needs additionally use the inequality \e {h104}.
		Thus, applying \lem {L5-1} for $A$ and $F=F_A$, we get a family $\ZG$ satisfying the conditions of the lemma. The family $\ZG$ will form the children collection of $A$, that is
		\begin{equation*}
		\Ch(A)=\ZG.
		\end{equation*}
		The relations 1) and 2) are immediate consequences of \e {a18} and \e {a20} respectively, while \e {a12} follows from condition \e {a21} of \lem {L5-1}. From \e {a22} it follows the existence of $\xi\in \tilde G\setminus F$ and by the definition \e {a26} we get \e {a13} (\e {h105} if $A=A_0$). Since $\xi\in \tilde G$, using \e {a21},\e {a61} and \e {a13}, for $f\in L^r(X)$ we get the inequality
		\begin{align*}
		\sup_{x\in G^{[2]}}|T(f\cdot \ZI_{\tilde G^{[1]}\setminus  G^{[3]}})(x)|&=\sup_{x\in G^{[2]}}|T(f\cdot\ZI_{A^{[3]}}\cdot\ZI_{\tilde G^{[1]}\setminus  G^{[3]}})(x)|\\
		&\le\Delta(G^{[2]},\tilde G) \cdot \langle f\cdot \ZI_{A^{[3]}}\rangle_{\tilde G^{[1]},r}\\
		&\lesssim \ZL \langle f\cdot \ZI_{A^{[3]}}\rangle^*_{\tilde G,r}\\
		&\le \ZL M_r(f\cdot \ZI_{A^{[3]}})(\xi)\\
		&\le \Gamma(f\cdot \ZI_{A^{[3]}})(\xi)=\Gamma f_A(\xi)\\
		&\le  \beta\cdot \langle f\rangle_{A^{[3]},r}\\
		&\lesssim  \ZL\lambda\cdot \langle f\rangle_{A^{[3]},r},
		\end{align*}
		which implies \e {a14}. From \e {a19} we get
		\begin{equation*}
		\mu\left(F\bigcap\left(A^{[1]}\setminus \bigcup_{G\in \Ch(A)} G\right)\right)=0,
		\end{equation*}
		therefore by the definition of $F$ we have \e {d23}. Hence, the properties of the family $\zA$ are satisfied. 
		
		Now we construct a sparse subfamily $\ZS\subset \zA$ consisting of countable collection of balls, which will satisfy the conditions of the theorem. We will do that by removing some elements of $\zA$. As we will see below, 
		removing an element $A\in \zA$, we also remove all the elements of its generation $\Gen (A)$. Thus one can easily check that the properties 1)-3) will hold during the whole process. 
		
		To start the description of the process we let $R=\ZK^2$, where $\ZK>1$ is the constant from \e {h13}. For $B\in \ZB$ denote
		\begin{equation*}
		\r(B)=[\log_R \mu(B)]
		\end{equation*}
		Observe that the collections of balls
		\begin{align*}
		\zA_k&=\{B\in \zA:\, \r(B)=k\}\\
		&=\left\{B\in \zA:\, R^{k}\le \mu(B)<R^{k+1}\right\},\quad k\le k_0,
		\end{align*}
		gives a partition of $\zA$, i.e. we have $\zA=\cup_{k\le k_0} \zA_k$, where $k_0=\r(A_0)$ and $A_{k_0}=\{A_0\}$.
		The reduction of the elements of $\zA$ will be realized in different stages. The content of $\zA_{k_0}$ will not be changed. In the $n$-th stage only the contents of the families $\zA_k$ with $k\le k_0-n$ can be changed. Besides, at the end of the $n$-th stage $\zA_{k_0-n}$ will be fixed and remain the same till the end of the process. Suppose by induction the $l$-th stage of reduction has been already finished and so the families $\zA_k$, $k=k_0,k_0-1,k_0-2,\ldots, k_0-l$ have already fixed. In the next $(l+1)$-th stage we will apply the following two procedures consecutively:
		\begin{procedure} Remove any element $G\in \zA_{k_0-l-1}$ together with all the elements of his generation $\Gen(G)$ if there exists a $B\in \zA$ satisfying the conditions  
			\begin{align}	
			&G^{[2]}\cap B\neq \varnothing ,\label{d29}\\
			&\r(\pr^k(G))\ll\r(B)\ll \r(\pr^{k+1}(G)),\label{d32}
			\end{align}
			for some integer $k\ge 0$.
		\end{procedure}
		\begin{remark}
			Observe that if an element $G$ is removed because of a ball $B$ satisfying the conditions \e {d29} and \e {d32} of Procedure 1, then we should have 
			\begin{equation*}
			\r(G)\ll\r(B)\ll\r(\pr(G))
			\end{equation*}
			that means \e {d32} can hold only with $k=0$. Indeed, the left inequality immediately follows from \e {d32}. To prove the right one, suppose to the contrary in \e {d32} we have $k\ge 1$. Denote 
			\begin{equation*}
			G'=\pr^k(G)\in \bigcup_{j=0}^{l}\zA_{k_0-j}.
			\end{equation*}
			Since $G'^{[2]}\supset G^{[2]}$ (see \e {a12}), we have $G'^{[2]}\cap B\neq \varnothing $. On the other hand  \e {d32} can be written by $\r(G')\ll\r(B)\ll\r(\pr(G'))$. We thus conclude that $G'$ satisfies the conditions of the Procedure 1, so $G'$  together with his generation $\Gen(G')$ (include $G$) had to be removed in one of the previous stages of the process, when $B$ was already fixed. This is a contradiction and so $k=0$. 
		\end{remark}
		\begin{procedure}
			Apply \lem {L1-1} to the rest of the elements $\zA_{k_0-l-1}$ having after Procedure 1. The application of lemma removes some elements of $\zA_{k_0-l-1}$. If an element $A$ is removed, then the generation $\Gen(A)$ will also be removed. 
		\end{procedure}
		\begin{remark}
			After the Procedure 2 the elements of $\zA_{k_0-l-1}$ become pairwise disjoint. Besides, we will have   
			\begin{equation}\label{a70}
			\bigcup_{G\in \zA_{k_0-l-1}(\text{\rm before Procedure 2})} G\subset \bigcup_{G\in \zA_{k_0-l-1}(\text{\rm after Procedure  2})}G^{[1]}.
			\end{equation}
		\end{remark}
		After these two procedures the family $\zA_{k_0-l-1}$ will be fixed. Hence, finishing the induction process, we get the final state of $\zA$ which will be denoted by $\ZD$. Since after Procedure 2 in $n$-th stage $\zA_{k_0-n}$ gets countable number of balls so the family $\ZD$ will also be countable at the end of whole process. 
		
		Now we shall prove that for an admissible constant $\lambda>0$ the family $\ZD$ is a union of two $1/2$-sparse collections of balls. For $A\in \ZD$ define
		\begin{equation}\label{d25}
		E(A)=A\setminus \bigcup_{G\in \ZD:\,\r(G)\ll\r(A)} G
		=A\setminus \bigcup_{G\in \ZD:\,G\cap A\neq\varnothing,\,\r(G)\ll\r(A)}G.
		\end{equation}
		Observe that 
		\begin{equation}\label{a27}
		E(A)\cap E(B)=\varnothing,\text { if } \r(A)\not \asymp \r(B) \text { or } \r(A)= \r(B).
		\end{equation}
		Indeed, take arbitrary $A,B\in \ZD$. If $\r(A)=\r(B)$, then the balls $A,B$ are result of the application of Procedure 2 (\lem {L1-1}). That means we  have $A\cap B=\varnothing$ and therefore according to \e {d25} it follows that $E(A)\cap E(B)=\varnothing$. If $\r(A)\gg\r(B)$, then $E(A)\cap B=\varnothing$
		immediately follows from the definition \e {d25} and so we will get again $E(A)\cap E(B)=\varnothing$. 
		To prove 
		\begin{equation}\label{a28}
		\mu(E(A))\ge \mu(A)/2
		\end{equation}
		take an arbitrary $A\in \ZD$ and denote
		\begin{equation*}
		\ZP=\{P\in \ZD:\, \r(P)\asymp\r(A),\,P^{[2]}\cap A\neq\varnothing\}.
		\end{equation*} 
		We have 
		\begin{equation}\label{h55}
		R^{-3}\cdot \mu(A)\le \mu(P)\le R^3\cdot \mu(A),\quad P\in \ZP.
		\end{equation}
		On the other hand
		\begin{equation*}
		\ZP\subset \zA_{l-1}\cup\zA_l\cup\zA_{l+1},
		\end{equation*}
		where $l=\r(A)$. Hence $\ZP$ consists of three families of pairwise disjoint balls. Thus, applying the remark after the \lem {L1-2}, we get 
		\begin{equation}\label{a84}
		\#\ZP\lesssim 1.
		\end{equation}
		Suppose that $G\in \ZD$ satisfies
		\begin{equation*}
		\r(G)\ll\r(A),\quad G\cap A\neq\varnothing,  
		\end{equation*}
		and so $G^{[2]}\cap A\neq\varnothing$. Since $G$ was not removed by Procedure 1, we have $\r(\pr^k(G))\asymp\r(A)$ for some integer $k\ge 1$. Denote $P=\pr^k(G)$. We have $\r(P)\asymp\r(A)$ and $P^{[2]}\supset G^{[2]}$ by \e {a12} and so $P^{[2]}\cap A\neq\varnothing$. This implies that $P\in\ZP$ and $G\in \Gen(P)$.  Hence, from \e {d25}, \e {a72}, \e {a84} and \e {h55} it follows that 
		\begin{align*}
		\mu(A\setminus E_A)&\le \mu\left(\bigcup_{G\in \ZD:\,G\cap A\neq\varnothing,\,\r(G)\ll\r(A)}G\right)\\
		&\le \mu\left(\bigcup_{P\in \ZP}\bigcup_{G\in \Gen(P)}G\right)\le \sum_{P\in \ZP}\sum_{k=1}^\infty \mu\left(\bigcup_{G:\,\pr^k(G)=P}G\right)\\
		&\lesssim \sum_{P\in \ZP}\sum_{k=1}^\infty \frac{\mu(P)}{\lambda^k}\lesssim \frac{\mu(A)}{\lambda}.  
		\end{align*}
		Thus for an admissible constant $\lambda$ we get 
		\e {a28}. From \e {a27} and \e {a28} we conclude that two families  
		\begin{equation*}
		\ZD_1=\{A\in\ZD:\, \r(A)\text { is odd }\},\quad \ZD_2=\ZD\setminus \ZD_1,
		\end{equation*}
		are sparse. Further for an $A\in \ZD$ we will need the bound
		\begin{align}
		&|\Gamma f_A(x)|\lesssim  \ZL\lambda\langle f\rangle_{A^{[3]},r},\label{d31}\\
		&\qquad\qquad \text{for a.a. } x\in A^{[1]}\setminus \bigcup_{G\in \ZD:\, \r(G)<\r(A)}G^{[1]},\nonumber
		\end{align}
		which is based on the inequality \e {d23}. It is enough to prove that 
		\begin{equation*}
		A^{[1]}\setminus  \bigcup_{B\in \ZD:\, \r(B)<\r(A)}B^{[1]}
		\subset A^{[1]}\setminus \bigcup_{G\in\zA:\, G\in\Ch(A)}G
		\end{equation*}
		or equivalently
		\begin{equation}\label{d28}
		D=\bigcup_{B\in \ZD:\, \r(B)<\r(A)}B^{[1]}
		\supset \bigcup_{G\in\zA:\, G\in\Ch(A)}G.
		\end{equation}
		Take $A\in\zA$ and arbitrary $G\in\Ch(A)$. We have $\r(G)<\r(A)$. In the case $G\in \ZD$, that is $G$ has not been removed during the Procedures 1 and 2, $G$ is an element of the left union of \e {d28} and so $G\subset D$. If $G\not\in \ZD$, then $G$ has been removed during the removal process. If $G$ was removed by an application of Procedure 1, then there exists a ball $B\in \ZD$ such that $G^{[2]}\cap B\neq \varnothing$ and $\r(G)\ll\r(B)\ll\r(\pr(G))=\r(A)$ (see remark after the Procedure 1). From the inequality $R\ge \ZK^3$ it follows that
		\begin{equation*}
		\mu(G^{[2]})\le \ZK^2\mu(G)\le \ZK^2\cdot \frac{\mu(B)}{\ZK^2}= \mu(B).
		\end{equation*}
		Thus we get $G\subset G^{[2]}\subset B^{[1]}$, which means $G\subset D$. If $G$ was removed by an application of Procedure 2, then according to \e {a70} we have $G\subset \cup_kB_k^{[1]}$ for a family of balls $B_k$ satisfying $\r(B_k)=\r(G)<\r(A)$ and so $B_k^{[1]}\subset D$. This again implies $G\subset D$ and so we get \e {d31}. To prove the theorem we need to prove
		\begin{align}
		|Tf(x)|&\lesssim  \ZL\lambda\cdot \ZA_{\ZS,r} f(x)\label{a73}\\
		&=\ZL\lambda \cdot \sum_{S\in\ZS}\left\langle f\right\rangle_{S,r}\cdot\ZI_{S}(x)\text { a.e. } x\in X,\nonumber
		\end{align}
		where $\ZS=\{S^{[3]}:\, S\in \ZD\}$ clearly is a union of two sparse collections. Observe that the set 
		\begin{equation*}
		E=\bigcap_{k\le k_0} \bigcup_{G\in\ZD:\, \r(G)\le k}A^{[1]}
		\end{equation*}
		has zero measure, since $\ZD$ consists of countable number of balls with bounded sum of their measures. Besides we can fix a set $F$ of zero measure such that for each $A\in \ZD$ inequality \e {d31} holds for any
		\begin{equation*}
		x\in \left(A^{[1]}\setminus \bigcup_{G\in \ZD:\, \r(G)<\r(A)}G^{[1]}\right)\setminus F.
		\end{equation*} 
		Hence, it is enough to prove the bound \e {a73} for arbitrary $x\in A_0\setminus (E\cup F)$. Observe that for such $x$ there exists a ball $A\in \ZD$ such that
		\begin{equation*}
		x\in \left(A^{[1]}\setminus \bigcup_{G\in \ZD:\,\r(G)<\r(A)}G^{[1]}\right),
		\end{equation*}  
		and so from \e {d31} we conclude 
		\begin{equation}\label{a79}
		|T f_A(x)|\le |\Gamma f_A(x)|\lesssim  \ZL\lambda \langle f\rangle_{A^{[3]},r}.
		\end{equation}
		According to the property 1) one can find a unique sequence of balls 
		$A_0,A_1,A_2,\ldots A_k=A$ in $\ZD$ such that $A_j=\pr(A_{j+1})$.
		According to the properties \e {a12}-\e {a14} there exist balls $\tilde A_j$ and points $\xi_j$, $j=0,1,\ldots ,k-1$,  such that
		\begin{align}
		&A_{j+1}^{[3]}\subset \tilde A_{j+1}^{[1]}\subset A_{j}^{[2]},\label{a85} \\
		&\Gamma f_A(\xi_j)\lesssim  \ZL\lambda\langle f\rangle_{A_j^{[3]},r},\quad \xi_j\in\tilde A_{j+1},\label{h105}\\	
		&|T(f\cdot \ZI_{\tilde A_{j+1}^{[1]}\setminus A_{j+1}^{[3]}})(t)|\lesssim  \ZL\lambda\langle f\rangle_{A_j^{[3]},r},\quad t\in A_{j+1}^{[2]}. \label{a31}
		\end{align}
		Since 
		\begin{equation*}
		x\in A^{[1]}=A_k^{[1]}\subset A_{j+1}^{[2]},
		\end{equation*}
		the condition \e {a31} holds for $t=x$. Thus, we get
		\begin{equation}\label{a74}
		T(f\cdot \ZI_{\tilde A_{j+1}^{[1]}\setminus A_{j+1}^{[3]}})(x)|\lesssim  \ZL\lambda \cdot \langle f\rangle_{A_j^{[3]}}.
		\end{equation}
		We claim that
		\begin{equation}\label{t20}
		|Tf_{A_j}(x)|\le C \ZL\lambda\cdot  \langle f\rangle_{A_j^{[3]},r}+|Tf_{A_{j+1}}(x)|,
		\end{equation}
		where $C>1$ is an admissible constant. Indeed, by T1)-condition and \e {h105}, we will have
		\begin{multline}\label{a75}
		\left|T(f_{A_j}\cdot \ZI_{X\setminus \tilde A_{j+1}^{[1]}})(x)-T(f_{A_j}\cdot \ZI_{X\setminus\tilde A_{j+1}^{[1]}})(\xi_j)\right|\lesssim \ZL \langle f_{A_j} \rangle^*_{\tilde A_{j+1},r}\\
		\le \ZL M_rf_{A_j}(\xi_j)\le \Gamma f_{A_j}(\xi_j)\lesssim\ZL \lambda \cdot \langle f\rangle_{A_j^{[3]},r}.
		\end{multline}
		Besides, from \e {a74} we get
		\begin{align}
		\left|T(f_{A_j}\cdot \ZI_{\tilde A_{j+1}^{[1]}\setminus A_{j+1}^{[3]}})(x)\right|&=\left|T(f\cdot \ZI_{\tilde A_{j+1}^{[1]}\setminus A_{j+1}^{[3]}})(x)\right|\label{a76}\\
		&\lesssim  \ZL\lambda\cdot \langle f\rangle_{A_j^{[3]},r}.\nonumber
		\end{align}
		From the definition of $f_{A_j}$ (see \e {h106}) and \e {a85} it follows that
		\begin{equation}\label{h108}
		f_{A_j}\cdot \ZI_{X\setminus \tilde A_{j+1}^{[1]}}=f\cdot \ZI_{A_j^{[3]}\setminus \tilde A_{j+1}^{[1]}}=f\cdot \ZI_{A_j^{[3]}}-f\cdot \ZI_{\tilde A_{j+1}^{[1]}}.
		\end{equation}
		Thus, applying \e {a5}, \e {a75}, \e {h105}, \e {a76} and \e {h108}, we conclude	
		\begin{align*}
		\left|Tf_{A_j}(x)\right|&=\left|Tf_{A_j}(x)\right| \\
		&\le\left|T(f_{A_j}\cdot \ZI_{X\setminus \tilde A_{j+1}^{[1]}})(x)\right|+\left|T(f_{A_j}\cdot \ZI_{\tilde A_{j+1}^{[1]}})(x)\right|\\
		&\le \left|T(f_{A_j}\cdot \ZI_{X\setminus \tilde A_{j+1}^{[1]}})(x)-T(f_{A_j}\cdot \ZI_{X\setminus \tilde A_{j+1}^{[1]}})(\xi_j)\right|+\left|T(f_{A_j}\cdot \ZI_{X\setminus \tilde A_{j+1}^{[1]}})(\xi_j)\right|\\
		&\qquad+\left|T(f_{A_j}\cdot \ZI_{\tilde A_{j+1}^{[1]}\setminus A_{j+1}^{[3]}})(x)\right|+\left|T(f_{A_j}\cdot \ZI_{A_{j+1}^{[3]}})(x)\right|\\
		&\le C \ZL\lambda\cdot\langle f\rangle_{A_j^{[3]},r}+\left|Tf_{A_j}(\xi_j)\right|\\
		&\qquad +\left|T(f\cdot \ZI_{\tilde A_{j+1}^{[1]}})(\xi_j)\right|+\left|Tf_{A_{j+1}}(x)\right|\\
		&\le C \ZL\lambda\cdot\langle f\rangle_{A_j^{[3]},r}\\
		&\qquad +2T^{**}f_{A_{j}}(\xi_j)+\left|Tf_{A_{j+1}}(x)\right|\\
		&\le C\ZL\lambda\cdot\langle f\rangle_{A_j^{[3]},r}\\
		&\qquad +4\Gamma f_{A_{j}}(\xi_j)+\left|Tf_{A_{j+1}}(x)\right|\\
		&\le C \ZL\lambda\cdot\langle f\rangle_{A_j^{[3]},r}+\left|Tf_{A_{j+1}}(x)\right|,
		\end{align*}
		where $C>0$ is an admissible constant that can vary in the above inequalities. Thus we get \e {t20}. Applying \e {t20} for each $j=0,1,2,\ldots, k-1$, \e {h106} and \e {a79}, we get
		\begin{align*}
		\left|Tf(x)\right|=\left|Tf_{A_0}(x)\right|&\le C \ZL\lambda\cdot \sum_{j=0}^{k-1}\langle f\rangle_{A_j^{[3]},r}+\left|T(f\cdot \ZI_{A^{[3]}})(x)\right|\\
		&\lesssim \ZL\cdot\sum_{j=0}^{k}\langle f\rangle_{A_j^{[3]},r}\\
		&\lesssim \ZL\cdot \ZA_{\ZS,r} f(x).
		\end{align*}
		completing the proof of \trm {T1}.
	\end{proof}
	\begin{proof}[Proof of \trm {T0}]
		\trm {T0} immediately follows from \trm {T1}, \trm {T5-1} and \trm {T6},
	\end{proof}
	
	\section{Weighted estimates in abstract measure spaces}\label{S6}
	\subsection{The general case}
	Let $w$ satisfy $A_p$-condition. We denote $q=\frac{p}{p-1}$ and let $\sigma= w^{-\frac{1}{p-1}}$ be the dual weight of $ w$. Note that if $ w\in A_p$, then $\sigma= w^{1/(1-p)}\in A_q$ and 
	\begin{equation}\label{h90}
	[\sigma]_{A_q}=[w]_{A_p}^{1/(p-1)}.
	\end{equation} 
	Besides we have 
	\begin{equation*}
	[w]_{A_p}=\sup_{B\in\ZB}\frac{ w(B)}{\mu(B)}\left(\frac{\sigma(B)}{\mu(B)}\right)^{p-1}.
	\end{equation*}
	The notation $dw$ in the integrals will stand for $wd\mu$, where $\mu$ is the basic measure.  Any weight $w$ defines a measure on the basic measurable space $(X,\ZM)$ and the $w$-measure of a set $E$ is defined as $w(E)=\int_Edw$. Everywhere below we denote by $c_p$ different constants depending only on $1<p<\infty$.  In this section we shall consider maximal functions with respect to different measures. So we denote the maximal function associated to a measure $\beta$ by
	\begin{equation*}
	M_\beta f(x)=\sup_{B\in \ZB:\, x\in B} \frac{1}{\mu(B)}\int_B|f(t)|d\beta(t).
	\end{equation*}  
	Recall that $\ZA_{\ZS}$ denotes the sparse operator corresponding to the case of $r=1$.
	\begin{lemma}\label{L10}
		Let $(X,\ZM,\mu)$ be a measure space with a ball-basis $\ZB$. If $\ZS\subset \ZB$  is a $\gamma$-sparse collection ($0<\gamma<1$) and the weight $ w$ satisfies the $A_p$ condition for $1<p\le 2$, then
		\begin{equation}\label{z47}
		\|\ZA_{\ZS}\|_{L^p( w)\to L^p( w)}\le \gamma^{-1}[ w]_{A_p}^{1/(p-1)}\|M_ w\|_{L^p( w)\to L^p( w)}^{1/(p-1)}\cdot \|M_\sigma\|_{L^p(\sigma)\to L^p(\sigma)}.
		\end{equation}
	\end{lemma}
	\begin{proof}
		Applying the H\"{o}lder inequality, for a measurable $E\in \ZM$ we have
		\begin{align}
		\mu(E)=&\int_E w^{1/p}\cdot  w^{-1/p}d\mu\label{h58}\\
		&\le\left(\int_E w d\mu\right)^{1/p}\cdot \left(\int_E w^{-\frac{1}{p-1}} d\mu\right)^{1/q}\nonumber\\
		&=( w(E))^{1/p}(\sigma(E))^{1/q}.\nonumber
		\end{align}
		Now suppose $\ZS$ is a sparse collection of balls and $f\in L^p(w)$ is positive. Using the inequality $(\sum a_k)^{p-1}\le \sum a_k^{p-1}$ for $1<p<2$, we obtain
		\begin{align}
		\|\ZA_{\ZS}f\|_{L^p( w)}^p&=\int_X\left((\ZA_{\ZS}f)^{p-1}\right)^\frac{p}{p-1} w d\mu\label{h60}\\
		&=\int_X\left(\left(\sum_{B\in\ZS}\langle f\rangle_B\cdot \ZI_{B}\right)^{p-1}\right)^{q} w d\mu\nonumber\\
		&\le\int_X\left(\sum_{B\in\ZS}\langle f\rangle_B^{p-1}\cdot \ZI_{B}\right)^{q} w d\mu\nonumber\\
		&=\left\|\sum_{B\in\ZS}\langle f\rangle_B^{p-1}\cdot \ZI_{B}\right\|_{L^q( w)}^q.\nonumber
		\end{align} 
		There is a function $g\in L^p(w)$ with $\|g\|_{L^p(w)}=1$ such that
		\begin{align*}
		\left\|\sum_{B\in\ZS}\langle f\rangle_B^{p-1}\cdot \ZI_{B}\right\|_{L^q( w)}&=\int_X\sum_{B\in\ZS}\langle f\rangle_B^{p-1}\cdot \ZI_{B}g wd\mu\\
		&= \sum_{B\in\ZS}\left(\frac{1}{\mu(B)}\int_Bfd\mu\right)^{p-1}\int_Bgdw\nonumber\\
		&=\sum_{B\in\ZS}\left(\frac{1}{\sigma(B)}\int_Bfd\mu\right)^{p-1}\times\nonumber\\
		&\times\frac{1}{ w(B)}\int_B|g|d w\cdot \frac{ w(B)}{\mu(B)}\left(\frac{\sigma(B)}{\mu(B)}\right)^{p-1}\cdot \mu(B).\nonumber
		\end{align*}
		Thus, applying \e {h58}, $A_p$-condition and H\"{o}lder's inequality, we get
		\begin{align}
		&\left\|\sum_{B\in\ZS}\langle f\rangle_B^{p-1}\cdot \ZI_{B}\right\|_{L^q( w)}\label{h61}\\
		&\quad \le \gamma^{-1}[ w]_{A_p}\sum_{B\in\ZS}\left(\frac{1}{\sigma(B)}\int_Bfd\mu\right)^{p-1}\times \frac{1}{ w(B)}\int_B|g|d w\cdot \mu(E_B)\nonumber\\
		&\quad\le\gamma^{-1}[ w]_{A_p}\sum_{B\in\ZS}\left(\frac{1}{\sigma(B)}\int_Bfd\mu\right)^{p-1}(\sigma(E_B))^{1/q}\times \nonumber \\
		&\qquad \qquad \times\frac{1}{ w(B)}\int_B|g|d w\cdot ( w(E_B))^{1/p}\nonumber\\
		&\quad\le\gamma^{-1}[ w]_{A_p}\left(\sum_{B\in\ZS}\left(\frac{1}{\sigma(B)}\int_Bf\sigma^{-1}d\sigma\right)^p\sigma(E_B)\right)^{1/q}\times\nonumber \\
		&\qquad\qquad \times \left(\sum_{B\in\ZS}\left(\frac{1}{ w(B)}\int_B|g|d w\right)^p\cdot  w(E_B)\right)^{1/p}.\nonumber
		\end{align}
		The last two factors can be estimated by the maximal functions $M_\sigma$ and $M_ w$ respectively. 
		Namely, for the second one we have
		\begin{align}
		\left(\sum_{B\in\ZS}\left(\frac{1}{ w(B)}\int_B|g|d w\right)^p\cdot  w(E_B)\right)^{1/p}&\le \left\|M_ w  g\right\|_{L^p(w)}\label{h68}\\
		&\le \|M_ w\|_{L^p( w)\to L^p( w)}\cdot \|g\|_{L^p( w)}\nonumber\\
		&=\|M_ w\|_{L^p( w)\to L^p( w)}.\nonumber
		\end{align}
		Similarly, the first factor is estimated by
		\begin{align}
		&\left(\sum_{B\in\ZS}\left(\frac{1}{\sigma(B)}\int_Bf\sigma^{-1}d\sigma\right)^p\sigma(E_B)\right)^{1/q}\label{h69}\\
		&\qquad\qquad\le \|M_\sigma\|_{L^p(\sigma)\to L^p(\sigma)}^{p/q}\cdot \|f\sigma^{-1}\|_{L^p(\sigma)}^{p/q}\nonumber\\
		&\qquad\qquad=\|M_\sigma\|_{L^p(\sigma)\to L^p(\sigma)}^{p/q}\cdot\left(\int_X f^p\sigma^{-p}\sigma d\mu\right)^{1/q}\nonumber\\
		&\qquad\qquad=\|M_\sigma\|_{L^p(\sigma)\to L^p(\sigma)}^{p/q}\cdot\left(\int_X f^pd w\right)^{1/q}\nonumber\\
		&\qquad\qquad=\|M_\sigma\|_{L^p(\sigma)\to L^p(\sigma)}^{p/q}\cdot\|f\|_{L^p( w)}^{p/q}.\nonumber
		\end{align}
		From \e {h60}, \e {h61}, \e {h68} and \e {h69} we immediately get \e {z47}.
	\end{proof}
	\begin{lemma}\label{L11}
		Let $(X,\ZM,\mu)$ be measure space with a ball-basis $\ZB$, $1<p,q<\infty $ and $p^{-1}+q^{-1}=1$. If $\ZS$  is a sparse collection and the weight $ w$ satisfies the $A_p$ condition, then
		\begin{equation}\label{h89}
		\|\ZA_{\ZS}\|_{L^p( w)\to L^p( w)}=\|\ZA_{\ZS}\|_{L^q(\sigma)\to L^q(\sigma)},
		\end{equation}
		where $\sigma$ is the dual weight of $ w$. 
	\end{lemma}
	\begin{proof}
		We have 
		\begin{equation*}
		\|\ZA_{\ZS}\|_{L^p( w)\to L^p( w)}=\sup_{f\in L^p( w),\,g\in L^q( w)}\int_X\ZA_{\ZS}f\cdot gd w.
		\end{equation*}
		By the duality argument for $f\in L^p( w)$ and $g\in L^q( w)$ we get the estimate
		\begin{align*}
		\int_X\ZA_{\ZS}f\cdot gd w&=\int_X\ZA_{\ZS}f\cdot g w d\mu=\sum_{B\in\ZS}\frac{1}{\mu(B)}\int_Bfd\mu\int_Bg w d\mu\\
		&=\int_X\ZA_{\ZS}(g w)\cdot f  d\mu=\int_X\frac{\ZA_{\ZS}(g w)}{ w}\cdot f   w d\mu=\int_X\frac{\ZA_{\ZS}(g w)}{ w}\cdot f  d w \\
		&\le \left\|\frac{\ZA_{\ZS}(g w)}{ w}\right\|_{L^q( w)}\|f\|_{L^p( w)}=\left(\int_X (\ZA_{\ZS}(g w))^q w^{-q}d w\right)^{1/q}\cdot \|f\|_{L^p( w)}\\
		&=\left(\int_X (\ZA_{\ZS}(g w))^q\sigma d\mu\right)^{1/q}\cdot \|f\|_{L^p( w)}\\
		&\le \|\ZA_{\ZS}\|_{L^q(\sigma)\to L^q(\sigma)}\|g w\|_{L^q(\sigma)}\|f\|_{L^p( w)}\\
		&=\|\ZA_{\ZS}\|_{L^q(\sigma)\to L^q(\sigma)}\left(\int_X(g w)^q\sigma d\mu\right)^{1/q} \|f\|_{L^p( w)}\\
		&=\|\ZA_{\ZS}\|_{L^q(\sigma)\to L^q(\sigma)}\|g\|_{L^q( w)}\|f\|_{L^p( w)},
		\end{align*}
		which implies $\|\ZA_{\ZS}\|_{L^p( w)\to L^p( w)}\le \|\ZA_{\ZS}\|_{L^q(\sigma)\to L^q(\sigma)}$. Similarly we have the reverse inequality and so \e {h89}.
	\end{proof}
	\begin{lemma}\label{L6}
		Let $(X,\ZM,\mu)$ be a measure space with a ball-basis $\ZB$ and $ w$ be a weight satisfying the $A_p$-condition, $1<p<\infty$, with respect to the measure $\mu$. Then for any balls $A,B$ with $A\subset B$ we have
		\begin{equation}\label{h10}
		\frac{ w(B)}{ w(A)}\le 2^p\cdot [w]_{A_p}\cdot \left(\frac{\mu(B)}{\mu(A)}\right)^p.
		\end{equation}
	\end{lemma}
	\begin{proof}
		Denote 
		\begin{equation}\label{h11}
		a=\frac{1}{\mu(A)}\int_A w d \mu=\frac{ w(A)}{\mu(A)}.
		\end{equation}
		By Chebishev's inequality we find
		\begin{equation*}
		\mu\{t\in A:\, w\le 2a\}>\frac{\mu(A)}{2}.
		\end{equation*} 
		Thus we get
		\begin{equation}\label{h64}
		\left(\int_B w^{1/(1-p)}\right)^{p-1}\ge \left(\frac{\mu(A)}{2}(2a)^{1/(1-p)}\right)^{p-1}
		=\frac{(\mu(A))^{p}}{2^{p} w(A)}
		\end{equation}
		and then by \e {h63} and \e {h64} we obtain
		\begin{align*}
		[w]_{A_p}&\ge \frac{1}{\mu(B)}\int_B w \cdot \left(\frac{1}{\mu(B)}\int_B w^{1/(1-p)}\right)^{p-1}\\
		&\ge \frac{ w(B)}{(\mu(B))^p}\cdot \left(\int_B w^{1/(1-p)}\right)^{p-1}\\
		&\ge 2^{-p}\cdot\frac{ w(B)}{ w(A)}\cdot\frac{( w(A))^p}{(\mu(B))^p}
		\end{align*}
		and so \e {h11}.
	\end{proof}
	\begin{lemma}\label{L13}
		Let $(X,\ZM,\mu)$ be a measure space with a ball-basis $\ZB$ and $ w$ be a weight, satisfying the $A_p$-condition, $1<p<\infty$. Then 
		the maximal function associated to the measure $ w$ satisfies the inequalities
		\begin{align}
		&\|M_ w\|_{L^1( w)\to L^{1,\infty}( w)}\le (2\ZK)^p\cdot [w]_{A_p},\label{h14}\\
		&\|M_ w\|_{L^{p'}( w)\to L^{p'}( w)}\le c(p,p')\cdot\ZK^{p/p'} [w]_{A_p}^{1/p'},\quad 1<p'< \infty.\label{h65}
		\end{align} 
		where the constant $c(p,p')$ depends on $p$ and $p'$.
	\end{lemma}
	\begin{proof}
		Denote
		\begin{equation*}
		E=\{ x\in B:\, M_ w f(x)>\lambda\}.
		\end{equation*}
		For any $x\in  E $ there exists a ball $B(x)\subset X$ such that
		\begin{equation*}
		x\in B(x),\quad \frac{1}{ w(B(x))}\int_{B(x)}|f|d w>\lambda.
		\end{equation*}
		Given ball $B\in \ZB$ consider the collection of balls $\{B(x):\, x\in E\cap B\}$. Apply \lem{L1-1}, we find a sequence of pairwise disjoint balls $\{B_k\}$ taken from this collection satisfying 
		\begin{equation*}
		E \cap B\subset \bigcup_kB_k^{[1]}.
		\end{equation*}
		Note that $B_k^{[1]}$ as usual is defined with respect to the measure $\mu$. From \lem {L6} we easily get $ w(B_k^{[1]})\le (2\ZK)^p\cdot [w]_{A_p} w(B_k)$ and hence
		\begin{align*}
		w\left(\cup_k B^{[1]}_k\right)	&\le \sum_k  w(B^{[1]}_k)\\
		&\le (2\ZK)^p\cdot [w]_{A_p} \sum_k  w(B_k)\\
		&\le (2\ZK)^p\cdot [w]_{A_p}\frac{1}{\lambda}\sum_k \int_{B_k}|f|d w\\
		&\le \frac{(2\ZK)^p\cdot [w]_{A_p}}{\lambda}\int_{X }|f|d w.
		\end{align*}
		Similarly as in the proof of \trm {T1-1}, thus we conclude 
		\begin{equation*}
		w^*(E)\le \frac{(2\ZK)^p\cdot [w]_{A_p}}{\lambda}\int_{X }|f|d w
		\end{equation*}
		and so \e {h14}. Applying \trm {M} (Marcinkiewicz interpolation theorem), we then get \e {h65}.
	\end{proof}
	\begin{theorem}\label{T9}
		If $\ZS$  is a $\gamma$-sparse collection and the weight $ w$ satisfies the $A_p$ condition for some $1<p<\infty$, then the corresponding sparse operator satisfies the bound
		\begin{equation}\label{h88}
		\|\ZA_{\ZS}f\|_{L^p( w)\to L^p( w)}\le c(p,\ZK)\cdot \gamma^{-1}[ w]_{A_p}^{\max\left\{\frac{p+2}{p(p-1)},\frac{3p-2}{p}\right\}}.
		\end{equation}
	\end{theorem}
	\begin{proof}
		First we suppose that $1<p\le 2$. Applying \lem {L10}, \lem {L13} and \e {h90}, we obtain
		\begin{align*}
		\|\ZA_{\ZS}\|_{L^p( w)\to L^p( w)}&\le c(p,\ZK)\gamma^{-1}[ w]_{A_p}^{1/(p-1)}\cdot [ w]_{A_p}^{1/p(p-1)}\cdot  [ \sigma]_{A_q}^{1/p}\\
		&=c(p,\ZK)\gamma^{-1}[ w]_{A_p}^\frac{p+2}{p(p-1)}
		\end{align*}
		and so \e {h88}. If $2<p<\infty$, then by \lem{L11} and \e {h90} we obtain
		\begin{multline*}
		\|\ZA_{\ZS}f\|_{L^p( w)\to L^p( w)}
		=\|\ZA_{\ZS}f\|_{L^q(\sigma)\to L^q(\sigma)}\\
		\le c(q,\ZK)\gamma^{-1}[\sigma]_{A_q}^{\frac{q+2}{q(q-1)}}=c(p,\ZK)\gamma^{-1}[ w]_{A_p}^\frac{3p-2}{p}.
		\end{multline*}
		Theorem is proved.
	\end{proof}
	Combining \trm {T1} with \trm {T9}, we obtain
	\begin{theorem}\label{T10}
		If $(X,\ZM,\mu)$ is a measure space with a ball-basis $\ZB$ and the operator $T\in\BO_\ZB(X)$ satisfies weak-$L^1$ inequality, then
		\begin{equation*}
		\|Tf\|_{L^p( w)\to L^p( w)}\le c(p,\ZK)(\ZL_1+\ZL_2+\|T\|_{L^1\to L^{1,\infty}}) [ w]_{A_p}^{\max\left\{\frac{p+2}{p(p-1)},\frac{3p-2}{p}\right\}}.
		\end{equation*}
	\end{theorem}
	\subsection{The case of Besicovitch condition}
	\begin{definition}
		Let $\ZB$ be a family of sets of an arbitrary set $X$. We say $\ZB$ satisfies the Besicovitch $D$-condition with a constant $D\in \ZN$, if for any collection $\ZA\subset \ZB$ one can find a subscollection $\ZA'\subset \ZA$ such that
		\begin{align*}
		&\bigcup_{A\in \ZA}A= \bigcup_{A\in \ZA'}A,\\
		&\sum_{ A\in \ZA'}\ZI_A(x)\le D.
		\end{align*}
		We say $\ZB$ is martingale system if $D=1$. 
	\end{definition}
	\begin{theorem}\label{T12}
		Let $(X,\ZM,\mu)$ be a measure space and the collection of measurable sets $\ZB\subset \ZM$ satisfy the Besicovitch $D$-condition. Then the maximal operator $M_\mu$ satisfies the bounds
		\begin{align}
		&\|M_\mu\|_{L^1(\mu)\to L^{1,\infty}(\mu)}\le D,\label{h16}\\
		&\|M_\mu\|_{L^p(\mu)\to L^p(\mu)}\le c_p\cdot D^{1/p},\quad 1<p<\infty.\nonumber
		\end{align} 
	\end{theorem}
	\begin{proof}
		Define
		\begin{equation*}
		E=\{ x\in X:\, M_\mu f(x)>\lambda\}.
		\end{equation*}
		For any $x\in  E $ there exists a set $B(x)\subset \ZB$ such that
		\begin{equation}
		\frac{1}{ w(B(x))}\int_{B(x)}|f|d w>\lambda,\quad x\in B(x).
		\end{equation}
		According to the Besicovitch condition there is a subcollection $\zA\subset \{B(x):\,x\in E\}$ such that
		\begin{align*}
		&\bigcup_{A\in \zA}A= \bigcup_{x\in E}B(x),\\
		&\sum_{ A\in \zA}\ZI_A(x)\le D.
		\end{align*} 
		Thus we get
		\begin{equation*}
		\mu^*(E)\le \sum_{A\in \zA}\mu(A)\le \frac{1}{\lambda}\sum_{A\in \zA}\int_{A}|f|\le \frac{D}{\lambda}\int_X|f|d\mu.
		\end{equation*} 
		The second inequality immediately follows from \e {h16}, according to Marcinkiewicz interpolation theorem (\trm {M}).
	\end{proof}
	The following theorem gives a sharp weighted estimate in general measure spaces with a ball basis satisfying the Besicovitch condition. 
	
	\begin{theorem}\label{T3}
		Let $\ZB$ be a family of sets in a measure space $(X,\ZM,\mu)$ satisfying the Besicovitch $D$-condition. If $\ZS$  is a $\gamma$-sparse collection and the weight $ w$ satisfies the $A_p$ condition for $1<p<\infty$, then
		\begin{equation}\label{h91}
		\|\ZA_{\ZS}f\|_{L^	p( w)\to L^p( w)}\lesssim c_p\gamma^{-1}D^{\max\{1/(p-1),p-1\}}\cdot [ w]_{A_p}^{\max\{1,1/(p-1)\}}.
		\end{equation}
	\end{theorem}
	\begin{proof}
		First suppose that $1<p\le 2$. Applying \lem {L10} and \trm {T12} we obtain
		\begin{align*}
		\|\ZA_{\ZS}\|_{L^p( w)\to L^p( w)}&\le c_p\gamma^{-1}[ w]_{A_p}^{1/(p-1)}\cdot D^{1/p(p-1)}\cdot  D^{1/p}\\
		&=c_p\gamma^{-1}[ w]_{A_p}^{1/(p-1)}\cdot D^{1/(p-1)}
		\end{align*}
		and so \e {h91}. In the case $2<p<\infty$ we use the same argument as in the proof of \trm {T9}.
	\end{proof}
	Applying \trm {T1} and \trm {T3} we immediately get the following
	\begin{theorem}\label{T2}
		Let a family of measurable sets $\ZB$ in a measure space $(X,\ZM,\mu)$ satisfy the Besicovitch $D$-condition and $ w$ be a $A_p$ weight with $1<p<\infty$. Then if an operator $T\in\BO_\ZB(X)$ satisfy weak-$L^1$ inequality, then
		\begin{equation*}
		\|T\|_{L^p( w)\to L^p( w)}\lesssim C(\ZL_1+\ZL_2+\|T\|_{L^1\to L^{1,\infty}}) \cdot[ w]_{A_p}^{\max\{1,1/(p-1)\}},
		\end{equation*}
		where $C$ is a constant depending on $p$ and the Besicovitch constant.
	\end{theorem}

	\section {Bounded oscillation operators on spaces of homogeneous type}\label{S9}
	
	\begin{definition}
		A quasimetric on a set $X$ is a function $\rho:X\times X\to [0,\infty)$ satisfying the conditions
		
		1) $\rho(x,y)\ge 0$ for every $(x,y)\in X$ and $\rho(x,y)= 0$ if and only if $x=y$,
		
		2) $\rho(x,y)=\rho(y,x)$ for every $x,y\in X$,
		
		3) $\rho(x,y)\le \ZD(\rho(x,z)+\rho(z,y))$ for every $x,y,z\in X$, where $\ZD>1$ is a fixed constant.
	\end{definition}
	Define the ball of a center $x$ and a radius $r$ by 
	\begin{equation*}
	B(x,r)=\{y\in X:\, \rho(x,y)<r\}, \quad x\in X,\quad 0<r< \infty
	\end{equation*}
	and denote by $\ZU(\rho)$ the family of all such balls, calling them $\rho$-balls. Quasimetric defines a topology, for which the $\rho$-balls form a base. In general, the balls need not to be open sets in this topology. For $B\in\ZU(\rho)$ denote by $c(B)$ and $r(B)$ respectively the center and the radius of $B$. For any $t>0$ we set $tB=B(c(B), tr(B))$.
	We define also an enlarged family of balls $\ZU'(\rho)$ as follows: if $\mu(X)=\infty$, then $\ZU'(\rho)$ coincides with $\ZU(\rho)$, in the case $\mu(X)<\infty$ we include in $\ZU'(\rho)$ additionally the set $X$. 
	\begin{definition} Let $\rho$ be a quasimetric on $X$ and $\mu$ be a positive measure defined on a $\sigma$-algebra $\ZM$ of subsets of $X$, containing the $\rho$-open sets and the $\rho$-balls. The collection $(X,\rho,\ZM,\mu)$ is said to be a space of homogeneous type if
		\begin{equation}\label{h101}
		\mu(2B)\le \ZH\cdot \mu(B)
		\end{equation}
		for any ball $B\in\ZU(\rho)$.
	\end{definition}
	Note that \e {h101} implies a more general inequality. Namely,
	\begin{equation}\label{z10}
	\mu(a\cdot B)\le \ZH(a)\cdot \mu(B),\quad a>0, 
	\end{equation}
	where $\ZH(a)$ is a constant depending on $a$ and $\ZH$. 
	From property 3) of quasimetric it easily follows that for any $B\in\ZU(\rho)$ it holds the inequality
	\begin{equation*}
	\diam B=\sup_{x,y\in B}\rho(x,y)\le 2\ZD\cdot r(B).
	\end{equation*}
	In this section the notation $a\lesssim b$ will stand for $a\le c\cdot b$, where $c>0$ is a constant depending on the constants $\ZH$ and $\ZD$ of the space homogeneous type .
	\begin{theorem}\label{T18}
		Let $(X,\rho,\ZM,\mu)$ be a space of homogeneous type such that $\ZU(\rho)$ satisfies the density condition. Then the enlarged family of balls $\ZU'(\rho)$ forms a ball-basis for the measure space $(X,\ZM,\mu)$ and satisfies the doubling condition. Besides, the hull ball of any $B=B(x_0,r)\in \ZU(\rho)$ has the form $B^*=B(x_0,R)$, $2r\le R\le \infty$.
	\end{theorem}
	The proof of the theorem is based on the following lemmas.
	\begin{lemma}\label{L23}
		If $(X,\rho,\ZM,\mu)$ is a space of homogeneous type, then for any point $x_0\in X$ and ball $G\in\ZU(\rho)$ we have
		\begin{equation}\label{z48}
		G\subset B(x_0,2\ZD^2(\dist(x_0,G)+r(G))).
		\end{equation}
	\end{lemma}
	\begin{proof}
		Fix a point $y\in G$ with $\rho(x_0,y)<2\dist(x_0,G)$. For arbitrary $x\in G$ we have
		\begin{align*}
		\rho(x,x_0)&\le \ZD(\rho(x,y)+\rho(y,x_0))\\
		&\le \ZD(2\ZD r(G)+2\dist(x_0,G))\le 2\ZD^2(\dist(x_0,G)+r(G))
		\end{align*}
		that means $x$ belongs to the right side of \e {z48}. 
	\end{proof}
	\begin{lemma}\label{L15}
		If $(X,\rho,\ZM,\mu)$ is a space of homogeneous type and the balls $B\in\ZU(\rho)$, $G_k\in\ZU(\rho)$, $k=1,2,\ldots$, satisfy the relations
		\begin{align}
		&B\cap G_k\neq \varnothing,\label{z50}\\
		& r(G_k)\to\infty \text { as } k\to\infty,\nonumber
		\end{align}
		then 
		\begin{equation*}
		\mu(X)\lesssim  \limsup_{k\to\infty }\mu(G_k).
		\end{equation*}
	\end{lemma}
	\begin{proof}
		Without loss of generality we can suppose that
		\begin{equation*}
		r(G_k)>r(B).
		\end{equation*}
		From \e{z50} it follows that $\dist(c(G_k), B)<r(G_k)$. Thus, applying \lem {L23}, we get
		\begin{align*}
		B(c(B), r(G_k))&\subset B(c(G_k),2\ZD^2 ( \dist(c(G_k), B)+r(G_k)))\\
		&\subset B(c(G_k),4\ZD^2 r(G_k)),
		\end{align*}
		and therefore by \e {z10} we obtain
		\begin{equation*}
		\mu(B(c(B),r(G_k)))\le  \ZH(4\ZD^2)\cdot \mu(G_k).
		\end{equation*}
		On the other hand, since $r(G_k)\to\infty$, we have $X=\cup_{k}B(c(B),r(G_k))$. Therefore we get
		\begin{equation*}
		\mu(X)=\lim_{k\to\infty }\mu(B(c(B),r(G_k))\lesssim\limsup_{k\to\infty }\mu(G_k).
		\end{equation*}
	\end{proof}

	\begin{lemma}\label{L9}
		Let $(X,\rho,\ZM,\mu)$ be a space of homogeneous type. Then for any $B=B(x_0,r)\in \ZU(\rho)$ there exists a ball $B^*=B(x_0,R)$  with $2r\le R\le \infty$ such that 
		\begin{align}
		&\qquad\qquad \mu(B^*)\lesssim \mu(B),\label{z14}\\
		&\bigcup_{A\in\ZU(\rho):\, \mu(A)\le 2\mu(B),\, A\cap B\neq\varnothing}A\subset B^*.\label{z13}
		\end{align}
	\end{lemma}
	\begin{proof}
		For a given $B\in \ZU(\rho)$ let $\zA$ be the family of balls $A\in\ZU(\rho)$ satisfying
		\begin{equation*}
		A\cap B\neq\varnothing,\quad \mu(A)\le 2\mu(B).
		\end{equation*}
		First suppose that
		\begin{equation*}
		\gamma=\sup_{A\in\zA}r(A)<\infty.
		\end{equation*}
		Applying \lem {L23}, for an arbitrary $A\in\zA$ we get 
		\begin{align*}
		A=B(c(A),r(A))&\subset B(c(B), 2\ZD^2(r(B)+r(A)))\\
		&\subset B(c(B), 4\ZD^2\gamma ).
		\end{align*}
		It is clear that $B^*=B(c(B), 4\ZD^2\gamma )$ satisfies \e {z13}. Take a ball $G\in \zA$ such that $r(G)>\gamma/2$. Again applying \lem {L23}, we get
		\begin{align*}
		B^*&=B(c(B), 4\ZD^2\gamma )\subset B(c(G),2\ZD(r(G)+4\ZD^2\gamma))\\
		&\subset B(c(G),10\ZD^3\gamma))\subset B(c(G),20\ZD^3r(G))\\
		&= 20\ZD^3\cdot G.
		\end{align*}
		Thus we conclude
		\begin{equation*}
		\mu(B^*)\le \mu(20\ZD^3\cdot G)\le \ZH(20\ZD^3)\mu(G)\lesssim \mu(B)
		\end{equation*}	
		that is just \e {z14}. Now consider the case $\gamma=\infty$. There is a sequence of balls $G_k\in\zA$ such that $r(G_k)\to\infty$. Applying \lem {L15}, we get
		\begin{equation*}
		\mu(X)\lesssim \limsup_{k\to\infty }\mu(G_k)\le 2\mu(B).
		\end{equation*}
		Obviously $B^*=B(x_0,\infty)=X$ satisfies \e {z14} and \e {z13}.
	\end{proof}
	
	\begin{proof}[Proof of \trm {T18}]
		We need to check conditions B1)-B4) of the definition of ball-basis. The conditions B1) and B2) immediately follows from the axioms of quasi-metric space and B4) follows from \lem{L9} and moreover for $B=B(x_0)\in \ZU(\rho)$ the hull ball $B^*$ has the form $B(x_0,R)$. The B3)-condition follows from the density property, since by \lem {L12}  those are equivalent. In order to prove the doubling condition, take a ball $A=B(x_0,r)$ such that $A^*=B(x,R)\neq X$. Denote	
		\begin{equation*}
		R'=\sup_{r'\ge R:\,B(x_0')=B(x_0,R)}r'.
		\end{equation*}
		Since $B(x_0,R)\neq X$, one can check that $R'<\infty$ and  
		\begin{equation*}
		A^*=B(x_0,R)=B(x_0,R')\subsetneq B(x_0,2R').
		\end{equation*} 
		Thus defining $B=B(x_0,2R')$, we get $A\subsetneq B$ and 
		\begin{equation*}
		\mu(B)=\mu(B(x_0,2R'))\lesssim \mu(B(x_0,R'))=\mu(A^*)\lesssim \mu(A),
		\end{equation*} 
		that proves the doubling condition. 
	\end{proof}
	
	\begin{theorem}\label{T17}
		Let $(X,\rho,\ZM,\mu)$ be a space homogeneous type satisfying the density condition. If $\ZS\subset \ZU(\rho)$  is a sparse collection of balls and the weight $ w$ satisfies the $A_p$-condition for $1<p<\infty$ (with respect to the family $\ZU(\rho)$), then the corresponding sparse operator satisfies the bound
		\begin{equation}\label{z2}
		\|\ZA_{\ZS}f\|_{L^p( w)\to L^p( w)}\lesssim c_p [w]_{A_p}^{\max\{1,1/(p-1)\}}.
		\end{equation}
	\end{theorem}
	The proof of this theorem is based on the Hyt\"{o}nen-Kairema \cite{HyKa} dyadic decomposition theorem, which reduces \trm {T17} to its martingale version (the case of $D=1$ in \trm {T3}).
	\begin{definition}
		Let $(X,\ZM,\mu)$ be a measure space. For two families of measurable sets $\ZB$ and $\ZB'$ we write $\ZB\prec  \ZB' $ if for any $B\in \ZB$ there exists $B'\in \ZB'$ such that
		\begin{equation*}
		B\subset B',\quad \mu(B')\le \gamma\mu(B),
		\end{equation*}
		where $\gamma>0$ is a constant. The minimum value of such constants $\gamma$ will be denoted by $\gamma(\ZB\prec  \ZB')$. If the relations $\ZB\prec\ZB' $ and $\ZB'\prec  \ZB $ hold simultaneously, then we write $\ZB\sim  \ZB'$ and denote
		\begin{equation*}
		\gamma(\ZB\sim  \ZB' )=\max\{\gamma(\ZB\prec  \ZB'),\gamma(\ZB'\prec  \ZB)\}.
		\end{equation*}  
	\end{definition}
	\begin{remark}
		One can verify that if for two families of measurable sets in $(X,\ZM,\mu)$ we have $\ZB\sim \ZB'$, then the $A_p$ characteristics with respect these families are equivalent. That is 
		\begin{equation*}
		0<c_1<\frac{\sup_{B\in \ZB}\left(\frac{1}{|B|}\int_Bw\right)\left(\frac{1}{|B|}\int_Bw^{-1/(p-1)}\right)^{p-1}}{\sup_{B\in \ZB'}\left(\frac{1}{|B|}\int_Bw\right)\left(\frac{1}{|B|}\int_Bw^{-1/(p-1)}\right)^{p-1}}<c_2, 
		\end{equation*}
		for some constants $c_1$ and $c_2$ depending on $\gamma(\ZB\sim  \ZB' )$.
	\end{remark}
	\begin{theorem}[Hyt\"{o}nen-Kairema \cite{HyKa}]\label{T15}
		If $(X,\rho,\ZM,\mu)$ is a space homogeneous type, then there exist martingale systems $\ZB_k\subset \ZM$, $k=1,2,\ldots,l$, such that 
		\begin{equation*}
		\ZU(\rho)\sim \ZB=\bigcup_{j=1}^l\ZB_j.
		\end{equation*}
		where $l$ and $\gamma(\ZU(\rho)\sim \ZB )$ are constants depending on $\ZH$ and $\ZD$.
	\end{theorem}
	\begin{proof}[Proof of \trm {T17}]
		Apply \trm {T15}. For every $B\in\ZU(\rho)$ there exists a set $Q(B)\in \ZB$ such that 
		\begin{equation}\label{z54}
		B\subset Q(B),\quad \mu(Q(B))\lesssim\mu(B). 
		\end{equation} 
		We shall consider the sparse operators 
		\begin{equation*}
		\ZA_kf(x)=\sum_{B\in \ZS:\,Q(B)\in \ZB_k}\langle f\rangle_{Q(B)}\ZI_{Q(B)}(x),\quad k=1,2,\ldots,l.
		\end{equation*} 
		From \e {z54} it follows that
		\begin{align}
		\ZA_{\ZS}f(x)&= \sum_{B\in\ZS}\langle f\rangle_B\ZI_B(x)\lesssim\sum_{B\in\ZS}\langle f\rangle_{Q(B)}\ZI_{Q(B)}(x)\label{z1}\\
		&\le \sum_{k=1}^l\sum_{B\in \ZS:\,Q(B)\in \ZB_k}\langle f\rangle_{Q(B)}\ZI_{Q(B)}(x)\nonumber\\
		&=\sum_{k=1}^l\ZA_kf(x).\nonumber
		\end{align}
		Since each $\ZB_k$ is martingale system, by \trm {T3} (for $D=1$) we conclude that
		\begin{equation*}
		\|\ZA_kf\|_{L^p(\omega)}\lesssim c_p [\omega]_{A_p}^{\max\{1,1/(p-1)\}}.
		\end{equation*}
		Combining this and \e {z1}, we get \e {z2}.
	\end{proof}
	
	Let $(X,\rho,\ZM,\mu)$ be a space of homogeneous type and $K(x,y):X\times X\to \ZR$ be a measurable function. 
	Given ball $\ZB\in \ZU(\rho)$ define the function
	\begin{align}
	&\phi_B(t)=\sup_{x,x'\in B,\, y\in X\setminus B(c(B),t)}|K(x,y)-K(x',y)|,\text { if } t\ge 2r(B),\label{h4}\\
	&\phi_B(t)=\phi_B(2r(B))\text { if } 0\le t<2r(B),\label{a90}
	\end{align}
	which is clearly decreasing on $[0,\infty)$. Denote
	\begin{align}
	&R=\sup_{B\in \ZB}\int_{X} \phi_B\left(\rho(y,c(B))\right)d\mu(y),\label{h5}\\
	&d_B=\sup_{x\in B,\, y\in X\setminus 2B}|K(x,y)| .\nonumber\\
	\end{align}
	\begin{definition}
		An operator $T:L^1(X)\to L^0(X)$ is said to be of Calder\'{o}n-Zygmund type if for any $B\in\ZU(\rho)$ it admits the representation
		\begin{equation}\label{h7}
		Tf(x)=\int_{X} K(x,y)f(y)d\mu(y)\text{ whenever } x\in B,\, \supp f\subset X\setminus 2B,
		\end{equation}
		where the kernel $K(x,y)$ satisfies the conditions
		\begin{equation}
		R<\infty,\quad d_B<\infty \text{ for any } B\in\ZU(\rho).
		\end{equation}
	\end{definition}
	
	\begin{theorem}\label{T11}
		If $(X,\rho,\ZM,\mu)$ is a space of homogeneous type such that $\ZU(\rho)$ satisfies the density condition, then any Calder\'{o}n-Zygmund type operator \e {h7} is $\BO$ operator with respect to the ball-basis $\ZU'(\rho)$. Moreover we have $\ZL_1(T)\le R$, where $R$ is \e {h5}.
	\end{theorem}
	\begin{proof}
		First note that \trm {T18} implies that $\ZU'(\rho)$ is a ball basis having the doubling property and for $B=B(x_0,r)\in \ZU(\rho)$ the hull ball $B^*$ has the form $B(x_0,R)$, $2r\le R<\infty$. Since $\ZU'(\rho)$ satisfies the doubling condition, according to \trm {T4}, we need to verify only T1)-condition.  Since $\phi_B(t)$ is decreasing, we can prove
		\begin{equation}\label{h3}
		\int_{X}\phi_B(\rho(y,c(B)))|f(y)|d\mu(y)\le R\cdot \langle f \rangle^*_B.
		\end{equation} 
		Indeed, one can easily find a step function $\psi(t)$ on $[0,\infty)$ such that
		\begin{align*}
		&	\phi_B(t)\le \psi(t)=\sum_{k=1}^\infty a_k\ZI_{[0,r_k]}(t),\, a_k>0,\, 2r(B)=r_1<r_2<\ldots,\\
		&	\int_{X} \psi\left(\rho(y,c(B))\right)d\mu(y)<\int_{X} \phi_B\left(\rho(y,c(B))\right)d\mu(y)+\delta\le R+\delta,
		\end{align*}
		where $\delta>0$ can be enough small.  We have
		\begin{align*}
		\int_{X}\phi_B(\rho(y,c(B)))|f(y)|d\mu(y)&\le 	\int_{X}\psi(\rho(y,c(B)))|f(y)|d\mu(y)\\
		&=\sum_{k=1}^\infty a_k \int_{B(c(B),r_k)}|f(y)|d\mu(y)\\
		&\le \langle f \rangle^*_B\sum_{k=1}^\infty a_k\mu(B(c(B),r_k))\\
		&= \langle f \rangle^*_B	\int_{X} \psi\left(\rho(y,c(B))\right)d\mu(y)\\
		&\le \langle f \rangle^*_B(R+\delta).
		\end{align*}
		Since $\delta>0$ is small enough, we get \e {h3}. Now take $B\in \ZU(\rho)$, $f\in L^1(X)$ and suppose that $x,x'\in B$.  From  \e {h4} it follows that
		\begin{equation*}
		|K(x,y)-K(x',y)|\le \phi_B(\rho(y,c(B)))|\text { whenever }y\in X\setminus 2B.
		\end{equation*}
		Thus, using  \e {h3} and the relation $B^*\supset 2B$, we get the bound
		\begin{align*}
		|T(f\cdot \ZI_{X\setminus B^*})(x)&-T(f\cdot \ZI_{X\setminus B^*})(x')|\nonumber\\
		&=\left|\int_{X\setminus B^*}(K(x,y)-K(x',y))f(y)d\mu(y)\right|\nonumber\\
		&\le\int_{X\setminus 2B}|K(x,y)-K(x',y)||f(y)|d\mu(y)\nonumber\\
		&\le \int_{X}\phi_B(\rho(y,c(B)))|f(y)|d\mu(y)\nonumber\\
		&\le R\cdot \langle f \rangle^*_B,\nonumber
		\end{align*}
		which gives T1)-condition.
	\end{proof}
	Let $(X,\rho,\ZM,\mu)$ be a space of homogeneous type and  $ \omega:[0,\infty) \to [0,\infty)$ be an increasing function satisfying $\omega(t+s)\le \omega(t)+\omega(s)$, $\omega(0)=0$, and the Dini condition
	\begin{equation}\label{h75}
	C_1=\int_0^1\frac{\omega(t)}{t}dt<\infty.
	\end{equation} 
	An operator $T:L^1(X)\to L^0(X)$ is said to be $\omega$-Calder\'{o}n-Zygmund operator if it has the representation \e {h7} and for any ball $B\in\ZU(\rho)$  we have 
	\begin{align}
	&\sup_{x\in B,\, y\in X\setminus B(c(B),t)}|K(x,y)|\le \frac{C_2}{\mu(B(c(B),t))},\label{h74}\\
	&\sup_{x,x'\in B,\, y\in X\setminus B(c(B),t)}|K(x,y)-K(x',y)|\le\frac{ \omega\left(\frac{r(B)}{t}\right)}{\mu(B(c(B),t))},\label{h77}\\
	&\sup_{y,y'\in B,\, x\in X\setminus B(c(B),t)}|K(x,y)-K(x,y')|\le  \frac{\omega\left(\frac{r(B)}{t}\right)}{\mu(B(c(B),t))},\label{h80}
	\end{align}
	where all these inequalities hold for any $t>2r(B)$.
	\begin{theorem}\label{T19}
		Let $(X,\rho,\ZM,\mu)$ be a space of homogeneous type such that $\ZU(\rho)$ satisfies the density condition. If $T$ is a $\omega$-Calder\'{o}n-Zygmund operator, then it is a $\BO$ operator with respect to the ball-basis $\ZU'(\rho)$. Moreover, we have the estimates 
		\begin{equation}\label{z15}
		\ZL_1(T)\lesssim C_1,\quad \ZL_2(T)\lesssim C_2.
		\end{equation}
	\end{theorem}
	\begin{proof}
		Taking into account \e {h77} and the definition of function $\phi_B$ in \e {h4}, \e {a90}, for every $B=B(x_0,r)\in\ZU(\rho)$ we have 
		\begin{align*}
		&\phi_B(t)\le \omega\left(\frac{r(B)}{t}\right)\frac{1}{\mu(B(c(B),t))}\text { if } t\ge 2r(B),\\
		&\phi_B(t)\le \omega\left(\frac{1}{2}\right)\frac{1}{\mu(2B)}\text { if } 0\le t<2r(B).
		\end{align*}
		Thus, applying doubling property \e {h101}, \e {h75} and the subadditivity of $\omega$, we get
		\begin{align*}
		\int_{X} \phi_B\left(\rho(y,c(B))\right)d\mu(y)&= \int_{2B} \phi_B\left(\rho(y,c(B))\right)d\mu(y)\\
		&\qquad +\sum_{k=1}^\infty\int_{(2^{k+1}B)\setminus (2^{k}B)} \phi_B\left(\rho(y,c(B))\right)d\mu(y)\\
		&\le \omega\left(\frac{1}{2}\right)+\sum_{k=1}^\infty \omega\left(2^{-k}\right)\frac{\mu(2^{k+1}B)-\mu(2^{k}B)}{\mu(2^{k}B)}\\
		&\lesssim \sum_{k=1}^\infty \omega(2^{-k})\lesssim  C_1.
		\end{align*} 
		Since this inequality holds for any ball $B$, applying \trm {T11}, we get the first estimate in \e {z15}. To estimate $\ZL_2(T)$  take a ball $A=B(x_0,r)\in \ZU(\rho)$ with $A^*\neq X$. According to \trm {T18}, $A^*=B(x_0,R)$, $R\ge 2r$. Denote
		\begin{equation*}
		L=\sup_{r'\ge R:\,B(x_0,r')=B(x_0,R)}r',
		\end{equation*}
		and let $B=B(x_0,2L)$. Since $B(x_0,R)\neq X$, we have $L<\infty$ and  
		\begin{equation*}
		A^*=B(x_0,R)=B(x_0,L)\subsetneq B(x_0,2L)=B.
		\end{equation*} 
		Thus we get $A\subsetneq B$ and
		\begin{equation}\label{h100}
		\mu(B^*)\lesssim\mu(B(x_0,2L))\le \ZH\mu(B(x_0,L))= \ZH\mu(A^*).
		\end{equation} 
		Since $r(A^*)\ge 2r(A)$, from \e {h74} we obtain
		\begin{equation*}
		\sup_{x\in A,\, y\in X\setminus A^*}|K(x,y)|\le \frac{C_2}{\mu(A^*)}.
		\end{equation*}
		Thus, using also \e {h100}, for any point $x\in A$ we get
		\begin{align*}
		|T(f\cdot\ZI_{B^*\setminus A^*})(x)|&\le \int_{B^*\setminus A^*}|K(x,y)||f(y)|dy\\
		&\le \frac{C_2}{\mu(A^*)}\int_{B^*}|f(y)|dy\\
		&\lesssim C_2 \langle f\rangle_{B^*}.
		\end{align*}
		Hence we obtain $\ZL_2(T)\lesssim C_2$ completing the proof of theorem. 
	\end{proof}
	Combining \trm {T1}, \trm {T17} and \trm {T19}, we immediately get
	\begin{theorem}\label{T13}
		Let $(X,\rho,\ZM,\mu)$ be a space of homogeneous type such that $\ZU(\rho)$ satisfies the density condition. If $T$ is a $\omega$-Calder\'{o}n-Zygmund operator and the weight $ w$ satisfies $A_p$ condition with respect to the ball-basis $\ZU(\rho)$, $1<p<\infty$, then we have
		\begin{equation}\label{z55}
		\|T\|_{L^p( w)\to L^p( w)}\le c_p(C_1+C_2+\|T\|_{L^1\to L^{1,\infty}}) [ w]_{A_p}^{\max\{1,1/(p-1)\}}.
		\end{equation} 
	\end{theorem}
	It is well known that any $\omega$-Calder\'{o}n-Zygmund operator, which is bounded on $L^2(X)$ satisfies the bound
	\begin{equation*}
	\|T\|_{L^1\to L^{1,\infty}}\lesssim \|T\|_{L^2\to L^2}.
	\end{equation*} 
	So in \e {z55} $\|T\|_{L^1\to L^{1,\infty}}$ can be replaced by $\|T\|_{L^2\to L^2}$. Note that the Hyt\"onen-Roncal-Tapiola \cite{Hyt6} inequality is the case of \e {z55} for the $\omega$-Calder\'{o}n-Zygmund operators on $\ZR^n$. Besides,  \e {z55} is a stronger version of the Anderson-Vagharshakyan \cite{Vag2} inequality, where the case of $\omega(t)=t^\delta$ was considered.
	\section{Other examples of BO operators}\label{S8}
	\begin{theorem}\label{T5}
		If $(X,\ZM,\mu)$ is a measure space with a ball-basis $\ZB$, then the maximal operator $M$ corresponding to $r=1$ in \e {1-1} is $\BO$ operator with respect to $\ZB$.
	\end{theorem}
	\begin{proof}
		In order to establish T1) condition we let $B$ be an arbitrary ball. Take two points $x,x'\in B$ and a nonzero function $f\in L^1(X)$ with 
		\begin{equation}\label{z56}
		\supp f\in X\setminus B^{[1]}.
		\end{equation} 
		Suppose that 
		\begin{equation}\label{z36}
		Mf(x)\ge Mf(x').
		\end{equation}
		We have $\langle f \rangle^*_B> 0$. Thus, by the definition of maximal operator we get 
		\begin{equation}\label{z37}
		Mf(x)\le \frac{1}{\mu(A)}\int_A|f|+\langle f \rangle^*_B
		\end{equation}
		for some ball $A\ni x$. If $\mu(A)\le  \mu(B)$, then by two balls relation we have $A\subset B^{[1]}$ and so by \e {z56} we get
		\begin{equation*}
		\frac{1}{\mu(A)}\int_A|f|=0.
		\end{equation*}
		Therefore according to \e {z36} and \e {z37} we have
		\begin{equation*}
		|Mf(x)-Mf(x')|=Mf(x)-Mf(x')\le \langle f \rangle^*_B.
		\end{equation*}
		If $\mu(A)> \mu(B)$, then we get $B\subset A^{[1]}$ and so
		\begin{align*}
		Mf(x)-Mf(x')&\le \frac{1}{\mu(A)}\int_A|f|+\langle f \rangle^*_B\\
		\lesssim  &\frac{1}{\mu(A^{[1]})}\int_{A^{[1]}}|f|+\langle f \rangle^*_B\lesssim \langle f \rangle^*_B.
		\end{align*} 
		This gives T1)-condition. To prove T2)-condition fix a ball $B$ and set
		\begin{align*}
		&\ZA=\{A\in\ZB:\,A\cap B\neq\varnothing,\, \mu(A)> \mu(B)\},\\
		&\gamma=\inf_{A\in \ZA}\mu(A).
		\end{align*}
		There exist a ball $A\in \ZA$ such that
		\begin{equation*}
		\gamma\le \mu(A)<2\gamma.
		\end{equation*}
		Define $B'=A^{[1]}$. One can check that
		\begin{equation}
		B\subsetneq A^{[1]}=B'.
		\end{equation}
		On the other hand for any function $f\in L^1(X)$ and any point $x\in B$ we have
		\begin{equation}\label{h22}
		M(f\cdot \ZI_{B'^{[1]}\setminus B^{[1]}})(x)=\frac{1}{\mu(C)}\int_C|f|\cdot \ZI_{B'^{[1]}\setminus B^{[1]}}+\delta,
		\end{equation}
		for some ball $C\ni x$ and a number $\delta>0$ that can be taken arbitrarily small. We can suppose that $\mu(C)>\mu(B)$, since otherwise we should have $C\subset B^{[1]}$, which will imply 
		\begin{equation*}
		\frac{1}{\mu(C)}\int_C|f|\cdot \ZI_{B'^{[1]}\setminus B^{[1]}}=0.
		\end{equation*}
		Hence, since we also have $C\cap B\neq\varnothing$, we get $C\in \ZA$. Thus we obtain $\mu(C)\ge \gamma$ and therefore
		\begin{equation*}
		\mu(C)>\frac{\mu(A)}{2}\ge \frac{\mu(A^{[2]})}{2\ZK^2}=\frac{\mu(B'^{[1]})}{2\ZK^2}.
		\end{equation*} 
		Hence we have
		\begin{equation}\label{h23}
		\frac{1}{\mu(C)}\int_C|f|\cdot \ZI_{B'^{[1]}\setminus B^{[1]}}\lesssim \frac{1}{\mu(B'^{[1]})}\int_{B'^{[1]}}|f|=\langle f\rangle_{B'^{[1]}}.
		\end{equation}
		Combining \e {h22} and \e {h23}, we get 
		\begin{equation*}
		M(f\cdot \ZI_{B'^{[1]}\setminus B^{[1]}	})(x)\lesssim \langle f\rangle_{B^{[1]}}
		\end{equation*}
		and so T2)-condition is proved.
	\end{proof}
	Thus, applying \trm {T1}, we get  
	\begin{theorem}\label{T21}
		Let $(X,\ZM,\mu)$ be a measure space with a ball-basis $\ZB$ and $ w$ be a $A_p$ weight with $1<p<\infty$. Then the maximal function \e {1-1} satisfies the bound
		\begin{equation*}
		\|M\|_{L^p( w)\to L^p( w)}\lesssim [w]_{A_p}^{\max\left\{\frac{p+2}{p(p-1)},\frac{3p-2}{p}\right\}}.
		\end{equation*}
		If in addition $\ZB$ satisfies the Besicovitch condition, then  
		\begin{equation*}
		\|M\|_{L^p( w)\to L^p( w)}\lesssim [w]_{A_p}^{\max\{1,1/(p-1)\}}.
		\end{equation*}
	\end{theorem}
	\begin{remark}
		\trm {T21} does not give the full weighted estimate like \e {z53}, which is known to be optimal for the maximal function in Euclidean spaces (\cite{Buck}). In the general case the optimality only occurs when $1<p\le 2$. The Buckley \cite{Buck} argument can be applied to get full bound \e {z53} in the case of Besicovitch condition.
	\end{remark}
	
	Another example of $\BO$ operator is the martingale transform. Let $(X,\ZM,\mu)$ be a measure space, and let $\{\ZB_n:\, n\in\ZZ\}$ be a collections of measurable sets such that
	\begin{enumerate}
		\item Each $\ZB_n$ is a finite or countable partition of $X$.
		\item For each $n$ and $A\in \ZB_n$  the set $A$ is the union of sets from $\ZB_{n+1}$.
		\item The collection $\ZB=\cup_{n\in\ZZ}\ZB_n$ generates the $\sigma$-algebra $\ZM$. 
		\item For any points $x,y\in X$ there is a set $A\in X$ such that $x,y\in A$.
	\end{enumerate}
	For a given $A\in \ZB$ let $\pr(A)$ (parent of $A$) be the minimal element of $\ZB$ satisfying $A\subsetneq \pr(A)$. One can easily check that $\ZB$ satisfies the ball-basis conditions B1)-B4). Moreover, for $A\in \ZB$ we can define
	\begin{equation}\label{z5}
	A^*=\left\{\begin{array}{lcl}
	&A\hbox{ if }& \mu(\pr(A))>2\mu(A),\\
	&\pr^n(A)\hbox{ if }& \mu(\pr^n(A))\le 2\mu(A) <\mu(\pr^{n+1}(A)),
	\end{array}
	\right.
	\end{equation}
	and take $\ZK=2$. Consider a function $f\in L^1(X)$. The martingale difference associated with $A\in \ZB$ is 
	\begin{equation*}
	\Delta_Af(x)=\sum_{B:\, \pr(B)=A}\left(\frac{1}{\mu(B)}\int_Bf-\frac{1}{\mu(A)}\int_Af	\right)\ZI_B(x).
	\end{equation*}
	The martingale transform operator is defined by
	\begin{equation*}
	Tf(x)=\sum_{A\in\ZB}\varepsilon_A\Delta_Af(x),
	\end{equation*}
	where $\varepsilon_A=\pm 1$ are fixed. 
	\begin{lemma}\label{L14}
		Any martingale transform $T$ satisfies T1)-condition and moreover, $\ZL_1(T)=0$.
	\end{lemma}
	\begin{proof}
		Take a function $f\in L^1(X)$ with 
		\begin{equation*}
		\supp f\in X\setminus A^{*}
		\end{equation*}
		and two points $x,x'\in A$. Observe that 
		\begin{align*}
		&\Delta_B f(x)=\Delta_B f(x'),\text { if } B\supset A^{*},\\
		&\Delta_B f(x)=\Delta_B f(x')=0, \text { if } B\subseteq A^{*}\text { or } B\cap A^{*}=\varnothing.
		\end{align*}
		Thus we have $\Delta_B f(x)=\Delta_B f(x')$ for any ball $B\in\ZB$ and so $Tf(x)=Tf(x')$. This implies $\ZL_1(T)=0$.
	\end{proof}
	\begin{lemma}\label{L5}
		If $T$ is a martingale transform, then for any ball $A\in\ZB$ we have
		\begin{equation}\label{f3}
		\sup_{x\in A,\,f\in L^1(X) }\frac{|T(f\cdot \ZI_{\pr(A)\setminus A})(x)|}{\langle f\rangle_{\pr(A)}}\le 2.
		\end{equation}
	\end{lemma}
	\begin{proof}
		Take a function $f\in L^1(X)$ with 
		\begin{equation*}
		\supp f\subset \pr(A)\setminus A.
		\end{equation*}
		Consider the sequence $A_k$, $k=1,2,\ldots$, defined by $A_0=A$, $A_{k+1}=\pr(A_k)$. For every point $x\in A$ we have
		\begin{align*}
		&\Delta_{A_k} f(x)=\left(\frac{1}{\mu(A_{k-1})}-\frac{1}{\mu(A_{k})}\right)\int_{\pr(A)}f\text { if } k>1,\\
		&\Delta_{B} f(x)=0, \text { if } B\subseteq A,\\
		&\Delta_{A_1} f(x)=-\frac{1}{\mu(A_1)}\int_{\pr(A)}f.
		\end{align*}
		Hence we get
		\begin{align*}
		|Tf(x)|&\le \sum_{k\ge 1}|\Delta_kf(x)|\le \frac{1}{\mu(A_1)}\int_{\pr(A)}|f|+\sum_{k>1}\left(\frac{1}{\mu(A_{k-1})}-\frac{1}{\mu(A_k)}\right) \int_{\pr(A)}|f|\\
		&\le \frac{2}{\mu(\pr(A))}\int_{\pr(A)}|f|\\
		&=2\langle f\rangle_{\pr(A)}
		\end{align*}
		Lemma is proved.
	\end{proof}
	\begin{lemma}\label{L16}
		If $T$ is a martingale transform, then $T$ satisfies T2)-condition and $\ZL_2(T)$ is bounded by an absolute constant.
	\end{lemma}
	\begin{proof}
		We need to prove the inequality
		\begin{equation}\label{h8}
		\sup_{x\in A,\,f\in L^1(X) }\frac{|T(f\cdot \ZI_{\pr(A)^{*}\setminus A^{*}})(x)|}{\langle f\rangle_{\pr(A)^{*}}}\le c,
		\end{equation}
		where $c>0$ is an absolute constant (see the definition of T2)-condition). 
		If $\mu(\pr(A))\le 2\mu(A)$, then applying \lem {L4} and \lem {L14}, the left hand side of \e {h8} can be estimated by $c\cdot \|T\|_{L^1\to L^{1,\infty}}$. Since $\ZK=2$, we can say that here $c>0$ is an absolute constant. It is well-known that $\|T\|_{L^1\to L^{1,\infty}}$ is estimated by an absolute constant too. This implies \e {h8}. In the case  $\mu(\pr(A))> 2\mu(A)$, applying $\ZK=2$, we obtain 
		\begin{align*}
		\sup_{x\in A,\,f\in L^1(X) }&\frac{|T(f\cdot \ZI_{\pr(A)^{*}\setminus A^{*}})(x)|}{\langle f\rangle_{\pr(A)^{*}}}\\
		&\le \sup_{x\in A,\,f\in L^1(X) }\frac{|T(f\cdot \ZI_{\pr(A)\setminus A^{*}})(x)|}{\langle f\rangle_{\pr(A)^{*}}}\\
		&\qquad\qquad+\sup_{x\in A,\,f\in L^1(X) }\frac{|T(f\cdot \ZI_{\pr(A)^{*}\setminus \pr(A)})(x)|}{\langle f\rangle_{\pr(A)^{*}}}\\
		&\le 2\sup_{x\in A,\,f\in L^1(X) }\frac{|T(f\cdot \ZI_{\pr(A)\setminus A})(x)|}{\langle f\rangle_{\pr(A)}}\\
		&\qquad\qquad+\sup_{x\in A,\,f\in L^1(X) }\frac{|T(f\cdot \ZI_{\pr(A)^{*}\setminus \pr(A)})(x)|}{\langle f\rangle_{\pr(A)^{*}}}.
		\end{align*} 
		The first terms in the last sum is estimated by \e {f3}. Now let us estimate the second term. Applying weak-$L^1$ inequality, for a $\lambda>0$ we can write
		\begin{align*}
		\mu\{x\in \pr(A):\,&|T(f\cdot \ZI_{\pr(A)^{*}\setminus \pr(A)})(x)|>\lambda \langle f\rangle_{\pr(A)^{*}} \}\\
		&\le \frac{\|T\|_{L^1\to L^{1,\infty}}}{\lambda \cdot \langle f\rangle_{\pr(A)^{*}}}\int_{{\pr(A)^{*}}}|f|\\
		&= \frac{\|T\|_{L^1\to L^{1,\infty}}\cdot \mu(\pr(A)^{*})}{\lambda}\le \frac{\mu(\pr(A))}{2},
		\end{align*}
		where the last inequality is obtained with a suitable absolute constant $\lambda>0$. This implies that the inequality 
		\begin{equation}\label{h102}
		|T(f\cdot \ZI_{\pr(A)^{*}\setminus \pr(A)})(x)|\le \lambda \langle f\rangle_{\pr(A)^{*}} 
		\end{equation}
		holds for some point $x\in \pr(A)$. Observe that the function $T(f\cdot \ZI_{\pr(A)^{*}\setminus \pr(A)})(x)$ is constant on $\pr(A)$ and it can be shown likewise the proof of \lem {L14}.  Hence we will have \e {h102} for any $x\in \pr(A)$. Thus we will give a bound of the second term by an absolute constant.
	\end{proof}
	\lem {L14} and \lem {L16} immediately imply
	\begin{theorem}\label{T20}
		The martingale transform is a $\BO$ operator with respect to the ball-basis $\ZB$. Moreover, $\ZL_1(T)=0$ and $\ZL_2(T)$ is bounded by an absolute constant.
	\end{theorem}
	Applying \trm {T1}, \trm {T3} and \trm {T20} we deduce the following results:
	\begin{theorem}[Lacey \cite{Lac1}]
		Let $T$ be a martingale transform. If $B\in\ZB$ and $f\in L^1(X)$, then there is a sparse operator $\ZA$ such that
		\begin{equation*}
		|Tf(x)|\le C\cdot \ZA f(x),\quad x\in B,
		\end{equation*}
		where $C$ is an absolute constant.
	\end{theorem}
	\begin{theorem}[Thiele, Treil and Volberg \cite {TTV}]
		If $T$ is a martingale transform and the weight $w$ satisfies $A_p$ condition with respect to the ball-basis $\ZB$, $1<p<\infty$, then we have
		\begin{equation*}
		\|T\|_{L^p( w)\to L^p( w)}\le c_p [ w]_{A_p}^{\max\{1,1/(p-1)\}}.
		\end{equation*} 
	\end{theorem}
	\section{Bounded oscillation of Carleson operators}
	Let $\{T_\alpha\}$ be a family of $\BO$ operators. In this section we prove that if the characteristic constants of operators $T_\alpha$ are uniformly bounded, then the domination operator
	\begin{equation}\label{aa4}
	Tf(x)=\sup_\alpha |T_\alpha f(x)|
	\end{equation} 
	is also $\BO$ operator. More precisely, we have
	\begin{theorem}\label{T16}
		If $(X,\ZM,\mu)$ is a measure space with a ball-basis $\ZB$. If a $\BO$-family of operators $\{T_\alpha\}$ satisfies weak-$L^r$ inequality, then the operator \e {aa4} satisfies the bounds  
		\begin{align}
		&\ZL_1(T)\lesssim \sup_\alpha \ZL_1(T_\alpha),\label{aa5}\\
		&\ZL_2(T)\lesssim \sup_\alpha\ZL_1(T_\alpha)+\sup_\alpha\ZL_2(T_\alpha)+\sup_\alpha\|T_\alpha\|_{L^r\to L^{r,\infty}}.\label{aa6}
		\end{align}
	\end{theorem}
	\begin{proof}
		Let $A\in\ZB$ be an arbitrary ball. Take two points $x,x'\in A$ and a nonzero function $f\in L^r(X)$ with 
		\begin{equation}\label{z56}
		\supp f\subset X\setminus A^{[1]}.
		\end{equation} 
		Suppose that 
		\begin{equation}\label{z36}
		Tf(x)\ge Tf(x').
		\end{equation}
		According to the definition of $T$, for any $\delta> 0$ there exists an index $\alpha$ such that
		\begin{equation}\label{z37}
		Tf(x)\le |T_\alpha f(x)|+\delta.
		\end{equation}
		On the other hand for the same $\alpha$ we have $Tf(x')\ge |T_\alpha f(x')|$. Thus, applying \e {z36}, \e {z37} and the localization property of $T_\alpha$, we obtain
		\begin{align*}
		|Tf(x)-Tf(x')|&=Tf(x)-Tf(x')\\
		&\le |T_\alpha f(x)|+\delta-|T_\alpha f(x')|\\
		&\le |T_\alpha f(x)-T_\alpha f(x')|+\delta\\
		&\le \ZL_1(T_\alpha)\langle f \rangle^*_{A,r}+\delta.
		\end{align*}
		Since $\delta>0$ can be taken enough small, we get \e {aa5}.	
		
		To prove \e {aa6} fix a ball $A$ ($A^*\neq X$) and consider the number
		\begin{equation*}
		\gamma=\inf_{B\in \ZB:\, B\supsetneq A}\mu(B).
		\end{equation*}
		Since $A\neq X$, from \lem {L1} it easily follows that the set of balls $B$ satisfying  $B\supsetneq A$ is nonempty. So there exists a ball $B\supsetneq A$ such that
		\begin{equation*}
		\gamma\le \mu(B)<2\gamma.
		\end{equation*}
		Take $f\in L^r(X)$ and a point $x\in A$. We have
		\begin{equation}\label{h22}
		T(f\cdot \ZI_{B^{[1]}\setminus A^{[1]}})(x)=|T_\alpha(f\cdot \ZI_{B^{[1]}\setminus A^{[1]}})(x)|+\delta,
		\end{equation}
		for an index $\alpha$, where $\delta>0$ can be arbitrarily small. Since $T_\alpha$ satisfies T2)-condition, there exists a ball $C\supsetneq A$ such that
		\begin{equation}\label{t6}
		T_\alpha(g\cdot \ZI_{C^{[1]}\setminus A^{[1]}})(x)\lesssim \ZL_2(T_\alpha)\langle g\rangle_{C^{[1]},r}
		\end{equation}
		holds for any $g\in L^r(X)$. Consider the function $g=f\cdot \ZI_{B^{[1]}\setminus A^{[1]}}$. If $\mu(C)>\mu(B)$, then we have $B\subset C^{[1]}$ and so $B^{[1]}\subset C^{[2]}$. Thus, applying \e {t6} and \lem {L4}, we obtain 
		\begin{align}
		|T_\alpha&g(x)|=|T_\alpha(g\cdot \ZI_{C^{[2]}\setminus A^{[1]}})(x)\label{t8}\\
		&\le |T_\alpha(g\cdot \ZI_{C^{[1]}\setminus A^{[1]}})(x)+|T_\alpha(g\cdot \ZI_{C^{[2]}\setminus C^{[1]}})(x)\nonumber\\
		&\le \ZL_2(T_\alpha)\langle g\rangle_{C^{[1]},r}+\sup_{x\in C,\, u\in L^r(X)}\frac{|T_\alpha (u\cdot \ZI_{C^{[2]}\setminus C^{[1]}})(x)|}{\langle u\rangle_{C^{[2]},r}}\cdot \langle g\rangle_{C^{[2]},r}\nonumber\\
		&\lesssim \ZL_2(T_\alpha)\left(\frac{1}{\mu(C^{[1]})}\int_{B^{[1]}}|f|^r\right)^{1/r}\nonumber\\
		&\quad +(\ZL_1(T_\alpha)+\|T_\alpha\|_{L^r\to L^{r,\infty}})\left(\frac{\mu(C^{[1]})}{\mu(C)}\right)^{1/r} \left(\frac{1}{\mu(C^{[2]})}\int_{B^{[1]}}|f|^r\right)^{1/r}\nonumber\\
		&\lesssim \sup_\alpha(\ZL_2(T_\alpha)+\ZL_1(T_\alpha)+\|T_\alpha\|_{L^r\to L^{r,\infty}})\cdot \langle f\rangle_{B^{[1]},r}.\nonumber
		\end{align}
		In the case $\mu(C)\le \mu(B)$ we have $C\subset B^{[1]}$, and since $C\supsetneq A$, we obtain 
		\begin{equation*}
		\mu(B)\gtrsim \mu(B^{[1]})\ge \mu(C)\ge \gamma \ge \frac{\mu(B)}{2}.
		\end{equation*} 
		Thus, again applying \lem {L4} and \e {t6}, we conclude 
		\begin{align}
		|T_\alpha(g)(x)|&\le |T_\alpha(g\cdot \ZI_{B^{[2]}\setminus C^{[1]}})(x)|+|T_\alpha(g\cdot \ZI_{C^{[1]}\setminus A^{[1]}})(x)|\label{t9}\\
		&\le \sup_{x\in C,\, u\in L^r(X)}\frac{|T_\alpha (u\cdot \ZI_{B^{[2]}\setminus C^{[1]}})(x)|}{\langle u\rangle_{B^{[2]},r}}\cdot \langle g\rangle_{B^{[2]},r}+\ZL_2(T_\alpha) \langle g\rangle_{C^{[1]},r}\nonumber\\
		&\lesssim (\ZL_1(T_\alpha))+\|T_\alpha\|_{L^r\to L^{r,\infty}})\langle g\rangle_{B^{[2]},r}+\ZL_2(T_\alpha)\langle g\rangle_{C^{[1]},r}\nonumber\\
		&\lesssim \sup_\alpha(\ZL_1(T_\alpha))+\|T_\alpha\|_{L^r\to L^{r,\infty}}+\ZL_2(T_\alpha))\cdot \langle f\rangle_{B^{[1]},r}.\nonumber
		\end{align}
		Observe that the admissible constants used in \e {t8} and \e {t9} do not depend on $f$, point $x$ the number $\delta$ from \e {h22}. Hence, since $\delta$ can be taken arbitrarily small, from \e {h22}, \e {t8} and \e {t9} we get the inequality
		\begin{align*}
		\sup_{x\in A,\,f\in L^r(X) }&\frac{T(f\cdot \ZI_{B^{[1]}\setminus A^{[1]}})(x)}{\langle f\rangle_{B^{[1]},r}}\\
		&\lesssim \ZL_1(T_\alpha))+\ZL_2(T_\alpha)+\|T_\alpha\|_{L^r\to L^{r,\infty}},
		\end{align*}
		which implies \e {aa6}.
	\end{proof}
	Let $T$ be a $\BO$ operator and  $\ZG=\{g_\alpha\} \subset L^\infty (X)$ be a family of functions such that
	\begin{equation}\label{t16}
	\beta=\sup_\alpha \|g_\alpha\|_\infty<\infty,\quad \|T\|_{L^r\to L^{r,\infty}}<\infty.
	\end{equation}
	One can easily check that the operators
	\begin{equation}
	T_\alpha f(x)= T(g_\alpha\cdot f)(x),
	\end{equation}
	are $\BO$ operator. Moreover, we have
	\begin{equation}\label{t17}
	\ZL_1(T_\alpha)\le \beta\ZL_1(T),\,  \ZL_2(T_\alpha)\le \beta\ZL_2(T),\, \|T_\alpha\|_{L^r\to L^{r,\infty}}\le \beta \|T\|_{L^r\to L^{r,\infty}}.
	\end{equation}
	Define the maximal modulation of the operator$T$ by
	\begin{equation}
	\quad T^\ZG f(x)=\sup_\alpha |T_\alpha f(x)|.
	\end{equation}
	According to \trm {T16} and relations \e {t17}, we conclude that $T^\ZG$ is also $\BO$ operator. Hence, applying \trm {T1}, we obtain
	\begin{theorem}\label{T22}
		If $T\in\BO_\ZB(X)$ satisfies \e {t16}, then for any function $f\in L^r(X)$ and a ball $B\in\ZB$ there exists a family of balls $\ZS$, which is a union of two sparse collections and 
		\begin{equation*}
		|T^\ZG f(x)|\lesssim \sup_\alpha \|g_\alpha\|_\infty (\ZL_1(T)+\ZL_2(T)+\|T^\ZG\|_{L^r\to L^{r,\infty}})\cdot \ZA_{\ZS,r}f(x),
		\end{equation*}
		for a.e. $x\in B$.
	\end{theorem}
	Weighted estimates of the maximal modulations of Calder\'{o}n-Zygmund operators on $\ZR^n$ (in particular Carleson or Walsh-Carleson operators) were considered in the papers \cite{PL, GMS}. \trm {T22} implies a pointwise sparse domination of such operators, which is the strongest version of the weighted norm domination of Carleson operators by sparse operators proved in \cite{PL}.

\end{document}